\documentclass[twoside,leqno,10pt, A4]{amsart}
\usepackage{amsfonts}
\usepackage{amsmath}
\usepackage{amscd}
\usepackage{amssymb}
\usepackage{amsthm}
\usepackage{amsrefs}
\usepackage{latexsym}
\usepackage{mathrsfs}
\usepackage{bbm}
\usepackage{amscd}
\usepackage{amssymb}
\usepackage{amsthm}
\usepackage{amsrefs}
\usepackage{latexsym}
\usepackage{mathrsfs}
\usepackage{bbm}
\usepackage{enumerate}
\usepackage{graphicx}
\usepackage{color}
\begin{document}

\newtheorem{theorem}[subsection]{Theorem}
\newtheorem{proposition}[subsection]{Proposition}
\newtheorem{lemma}[subsection]{Lemma}
\newtheorem{corollary}[subsection]{Corollary}
\newtheorem{conjecture}[subsection]{Conjecture}
\newtheorem{prop}[subsection]{Proposition}
\newtheorem{defin}[subsection]{Definition}

\numberwithin{equation}{section}
\newcommand{\mr}{\ensuremath{\mathbb R}}
\newcommand{\mc}{\ensuremath{\mathbb C}}
\newcommand{\dif}{\mathrm{d}}
\newcommand{\intz}{\mathbb{Z}}
\newcommand{\ratq}{\mathbb{Q}}
\newcommand{\natn}{\mathbb{N}}
\newcommand{\comc}{\mathbb{C}}
\newcommand{\rear}{\mathbb{R}}
\newcommand{\prip}{\mathbb{P}}
\newcommand{\uph}{\mathbb{H}}
\newcommand{\fief}{\mathbb{F}}
\newcommand{\majorarc}{\mathfrak{M}}
\newcommand{\minorarc}{\mathfrak{m}}
\newcommand{\sings}{\mathfrak{S}}
\newcommand{\fA}{\ensuremath{\mathfrak A}}
\newcommand{\mn}{\ensuremath{\mathbb N}}
\newcommand{\mq}{\ensuremath{\mathbb Q}}
\newcommand{\half}{\tfrac{1}{2}}
\newcommand{\f}{f\times \chi}
\newcommand{\summ}{\mathop{{\sum}^{\star}}}
\newcommand{\chiq}{\chi \bmod q}
\newcommand{\chidb}{\chi \bmod db}
\newcommand{\chid}{\chi \bmod d}
\newcommand{\sym}{\text{sym}^2}
\newcommand{\hhalf}{\tfrac{1}{2}}
\newcommand{\sumstar}{\sideset{}{^*}\sum}
\newcommand{\sumprime}{\sideset{}{'}\sum}
\newcommand{\sumprimeprime}{\sideset{}{''}\sum}
\newcommand{\sumflat}{\sideset{}{^\flat}\sum}
\newcommand{\shortmod}{\ensuremath{\negthickspace \negthickspace \negthickspace \pmod}}
\newcommand{\V}{V\left(\frac{nm}{q^2}\right)}
\newcommand{\sumi}{\mathop{{\sum}^{\dagger}}}
\newcommand{\mz}{\ensuremath{\mathbb Z}}
\newcommand{\leg}[2]{\left(\frac{#1}{#2}\right)}
\newcommand{\muK}{\mu_{\omega}}
\newcommand{\thalf}{\tfrac12}
\newcommand{\lp}{\left(}
\newcommand{\rp}{\right)}
\newcommand{\Lam}{\Lambda_{[i]}}
\newcommand{\lam}{\lambda}
\newcommand{\af}{\mathfrak{a}}
\newcommand{\sw}{S_{[i]}(X,Y;\Phi,\Psi)}
\newcommand{\lz}{\left(}
\newcommand{\pz}{\right)}
\newcommand{\bfrac}[2]{\lz\frac{#1}{#2}\pz}
\newcommand{\odd}{\mathrm{\ primary}}
\newcommand{\even}{\text{ even}}
\newcommand{\res}{\mathrm{Res}}

\theoremstyle{plain}
\newtheorem{conj}{Conjecture}
\newtheorem{remark}[subsection]{Remark}

\makeatletter
\def\widebreve{\mathpalette\wide@breve}
\def\wide@breve#1#2{\sbox\z@{$#1#2$}%
     \mathop{\vbox{\m@th\ialign{##\crcr
\kern0.08em\brevefill#1{0.8\wd\z@}\crcr\noalign{\nointerlineskip}%
                    $\hss#1#2\hss$\crcr}}}\limits}
\def\brevefill#1#2{$\m@th\sbox\tw@{$#1($}%
  \hss\resizebox{#2}{\wd\tw@}{\rotatebox[origin=c]{90}{\upshape(}}\hss$}
\makeatletter

\title[First moment of central values of Hecke $L$-functions with Fixed Order Characters]{First moment of central values of Hecke $L$-functions with Fixed Order Characters}


\author[P. Gao]{Peng Gao}
\address{School of Mathematical Sciences, Beihang University, Beijing 100191, China}
\email{penggao@buaa.edu.cn}

\author[L. Zhao]{Liangyi Zhao}
\address{School of Mathematics and Statistics, University of New South Wales, Sydney, NSW 2052, Australia}
\email{l.zhao@unsw.edu.au}

\begin{abstract}
 We evaluate asymptotically the smoothed first moment of central values of families of quadratic, cubic, quartic and sextic Hecke $L$-functions over various imaginary quadratic number fields of class number one, using the method of double Dirichlet series. In particular, we obtain asymptotic formulas for the quadratic families with error terms of size $O(X^{1/4+\varepsilon})$ under the generalized Riemann hypothesis.
\end{abstract}

\maketitle



\section{Introduction}
\label{sec 1}

The method of multiple Dirichlet series has been shown to be an elegant and powerful tool in analytical number theory with many
applications.  Among them, the most notable is their use in the study of moments of central values of families of $L$-functions.  In \cite{DoHo}, D. Goldfeld and J. Hoffstein utilized this approach to investigate the first moment of the family of quadratic Dirichlet $L$-functions. A systematic development of this method is given
by A. Diaconu, D. Goldfeld and J. Hoffstein in \cite{DGH}. \newline

   The effectiveness of the method of multiple Dirichlet series depends heavily on the analytic properties of the underlying series.  These properties are usually established using a sufficient number of functional equations for the multiple Dirichlet series of $n$ complex variables under consideration to acquire meromorphic continuation of the series to $\mc^n$. As this process often requires one to modify the objects
involved by attaching some correction factors, it is thus desirable to seek for ways that allow one to treat the series directly. In his study
of the ratios conjecture for quadratic Dirichlet $L$-functions using multiple Dirichlet series in \cite{Cech1},  M. \v Cech
established the functional equation for a general (not necessarily primitive) quadratic Dirichlet $L$-function in terms of Dirichlet series
whose coefficients are generalized Gauss sums. This enables him to meromorphically continue the underlying multiple Dirichlet series
to a suitable large region without altering it. \newline

  In \cite{G&Zhao2023-3}, the authors used the afore-mentioned approach of \v Cech to evaluate the first moment of central values of a family of quadratic
Dirichlet $L$-functions, obtaining an asymptotic formula under the generalized Riemann hypothesis (GRH) with an error term that is consistent with
the conjecture given by D. Goldfeld and J. Hoffstein \cite{DoHo}. It is the aim of this paper to further apply these ideas to
study the Hecke $L$-functions of fixed orders over number fields.  This is a rich subject previously investigated in \cites{FaHL, FHL,
Luo, Diac, BFH05, DT05, G&Zhao1, G&Zhao3, G&Zhao2020,G&Zhao2023-4,G&Zhao2023-5} using various methods, including multiple Dirichlet series
and approximate functional equations. The families of Hecke $L$-functions we study in this paper have the feature that their members are attached to characters not necessarily primitive, so that our treatments allow us to obtain better error terms in the resulting asymptotic formulas. \newline

  To state our results, we write $\mathcal{O}_K, U_K$ and $D_K$ for the ring of integers, the group of units and the discriminant of  an
arbitrary number field $K$, respectively. Further, let $N(n)$ and ${\rm Tr}(n)$ stand for the norm and the trace of any $n \in \mathcal{O}_K$,
respectively. We reserve the letter $\varpi$ for a prime element in $\mathcal O_K$, by which we mean that the ideal $(\varpi)$ generated by
$\varpi$ is a prime ideal. Also, we write $\zeta_K(s), L(s, \chi)$ for the Dedekind zeta function of $K$ and $L$-function attached to a Hecke
character $\chi$, respectively.  Moreover, let $r_K$ denote the residue of $\zeta_K(s)$ at $s = 1$. \newline

  In this paper, we consider the case when $K$ is an imaginary quadratic number field. It is well-known that $K$ can be written as $\mq(\sqrt{d})$ such that $d$ is a negative, square-free integer.  Further $K$ is of class number one (see \cite[(22.77)]{iwakow}) if and only if $d$ takes values from the set
\begin{align} \label{dvalue}
\mathcal{S} = \{-1, -2, -3, -7,-11,-19,-43,-67,-163 \}.
\end{align}

In the ramainder of this section, we consider the case when $K=\mq(\sqrt{d})$ with $d \in \mathcal{S}$.  Write $\chi_{j, m} =\leg {\cdot}{m}_j$, $\chi^{(m)}_{j}=\leg {m}{\cdot}_j$ for the $j$-th order residue symbols defined Section \ref{sec2.4} for certain $m \in \mathcal O_K$ and $j \in \{2, 3, 4, 6\}$ depending on $K$.  Note that $\chi^{(m)}_{j}$ is a Hecke character for every $0 \neq m \in \mathcal O_K$ whenever it is defined while $\chi_{j,m}$ is not since its values may vary at the units of $\mathcal O_K$. To remedy this, we treat the case $j=2$ and the case $j>2$ in two ways. \newline

For $j=2$, we observe that $U_K=\{\pm 1\}$ if $d \in \mathcal{S} \setminus \{ -1, -3 \}$, while $U_K$ is a cyclic group generated by $i$ for $K=\mq(i)$ and $U_K$ is a product of two cyclic groups generated by
$-1$ and $\omega:=(-1+\sqrt{-3})/2$ for $K=\mq(\sqrt{-3})=\mq(\omega)$. As $\omega^3=1$, we see that $\chi_{2,m}(\omega)=1$ for any $(m,2)=1$ in the latter case. We now set
\begin{align} \label{BKCKdef}
 B_K   =\begin{cases}
    2D_K \qquad & 2|d, \\
    \frac {(1-d)D_K}{2} \qquad & 2\nmid d,
\end{cases}
\quad \mbox{and} \quad
 C_K =\begin{cases}
   \displaystyle \{\chi^{(B^2_K)}_2,  \chi^{(iB^2_K)}_2 \}, \qquad & d=-1, \\
   \displaystyle  \{\chi^{(B^2_K)}_2,  \chi^{(-B^2_K)}_2 \}, \qquad & \text{otherwise}.
\end{cases}
\end{align}
  Then we observe, as the characters $\chi^{(B^2_K)}_2$ is trivial modulo $B_K$, that the sum $\frac 12\sum_{\chi \in C_K}\chi(m)$ allows us to pick up the elements $m \in \mathcal O_K$ such that $(m, B_K)=1$ and that $\chi_{2,m}$ is trivial on $U_K$. \newline

  For $j>2$, we note that from our discussions in Sections \ref{sec2.4} and \ref{sect: Lfcn} that the symbol $\chi_{j,m}$ is only defined for $j \in \{3, 4, 6\}$ on $K=\mq(i)$ or $K=\mq(\sqrt{-3})$. Further, $\chi_{j,m}$ is a Hecke character of trivial infinite
type for $m \equiv 1 \bmod S_{K,j}$, where $S_{K,j}$ can be taken to be any positive rational integer that is divisible by $j^2$. For
simplicity, we fix $S_{K,j}=144$ throughout the paper. We shall determine the condition $m \equiv 1 \bmod S_{K,j}$ using ray class group
characters. Here we call that for any integral ideal $\mathfrak{m} \in O_K$, the ray class group $h_{\mathfrak{m} }$ is defined to be $I_{\mathfrak{m} }/P_{\mathfrak{m} }$, where $I_{\mathfrak{m} } = \{
\mathcal{A} \in I : (\mathcal{A}, \mathfrak{m} ) = 1 \}$ and $P_{\mathfrak{m} } = \{(a) \in P : a \equiv 1 \bmod \mathfrak{m}  \}$ with $I$ and $P$ denoting the group of
fractional ideals in $K$ and the subgroup of principal ideals, respectively. \newline

 For any function $f$, let $\widehat f$ stand for its Mellin transform. Let $\Phi(x)$ be a fixed non-negative, smooth function compactly
supported on the set of positive real numbers ${\mr}_+$. In this paper, we evaluate the smoothed first moment of families of Hecke
$L$-functions averaged over all $j$-th order residue symbols for various $j$. For the quadratic case, our result is as follows.
\begin{theorem} \label{Thmfirstmoment}
With the notation as above and assuming the truth of GRH, for any imaginary quadratic number field $K$, we have for $1/2>\Re(\alpha)>0$, all large
$X$ and any $\varepsilon>0$,
\begin{align} \label{FirstmomentSmoothed}
\begin{split}
 \frac {1}{2}\sum_{\chi \in C_{K}} \sum_{0 \neq n \in \mathcal O_K} & L(\tfrac{1}{2}+\alpha, \chi \cdot \chi^{(n)}_{2})\Phi \left( \frac
{N(n)}X \right) \\
=& X\widehat \Phi(1)|U_K|r_K \frac{\zeta_K(1+2\alpha)}{\zeta_K(2+2\alpha)}\prod_{\varpi | B_K}\Big(1+\frac 1{N(\varpi)^{1+2\alpha}}\Big
)\Big(1-\frac 1{N(\varpi)^{2+2\alpha}}\Big )^{-1} \\
& \hspace*{1cm}+ X^{1-\alpha}\widehat \Phi(1-\alpha)\frac {|U_K|r_K}{2} \Big(\frac {2\pi}{\sqrt{|D_K|}}\Big )^{2\alpha} \frac{\Gamma (1-2\alpha)\Gamma (
\alpha)}{\Gamma(1-\alpha)\Gamma (2\alpha)}\cdot\frac{\zeta_K(1-2\alpha)}{\zeta_K(2)}\prod_{\varpi | B_K}\Big(1+\frac 1{N(\varpi)}\Big
)^{-1} \\
& \hspace*{1cm} +O\lz(1+|\alpha|)^{1+\varepsilon}|\alpha-1/2|^{-1}X^{1/4-\Re(\alpha)+\varepsilon}\pz.
\end{split}
\end{align}
\end{theorem}

   As the error term in \eqref{FirstmomentSmoothed} is uniform for positive $\alpha$ that is near $0$, we deduce, upon taking $\alpha \rightarrow 0^+$ and using the Laurent series of $\zeta(1+2\alpha)$ and $\zeta(1-2\alpha)$ centered at $\alpha=0$, the following result on the smoothed first moment of central values of quadratic families of Hecke $L$-functions.

\begin{corollary} \label{Thmfirstmomentatcentral}
Under the same assumptions and conditions of Theorem~\ref{Thmfirstmoment}, we have
\begin{align} \label{Asymptotic for first moment at central}
\begin{split}			
	& 	\frac {1}{2}\sum_{\chi \in C_K} \sum_{0 \neq n \in \mathcal O_K} L(\tfrac{1}{2}, \chi \cdot \chi^{(n)}_{2})\Phi \left( \frac {N(n)}X
\right)   = XQ_K(\log X)+O\lz X^{1/4+\varepsilon}\pz.
\end{split}
\end{align}
  where $Q_K$ is a linear polynomial whose coefficients depend only on $K$, $\widehat \Phi(1)$ and $\widehat \Phi'(1)$.
\end{corollary}

  We point out here that the error term in \eqref{Asymptotic for first moment at central} above matches with the conjectured size for the case
of quadratic Dirichlet $L$-functions given in \cite{DoHo} and our proof of Theorem \ref{Thmfirstmoment} implies that
\eqref{Asymptotic for first moment at central} holds with the error term $O\lz  X^{1/2+\varepsilon}\pz$ unconditionally without resorting to GRH. We omit the explicit
expression of $Q_K$ here since our main focus is the error term. Note also that a result that is similar to Corollary
\ref{Thmfirstmomentatcentral} is given in \cite[Theorem 1.1]{G&Zhao2023-5}.  The proofs of both Theorems~\ref{Thmfirstmoment} and
\cite[Theorem 1.1]{G&Zhao2023-5} use double Dirichlet series while our proof of Theorem \ref{Thmfirstmoment} differs from that of
\cite[Theorem 1.1]{G&Zhao2023-5} in the way that instead of seeking to develop enough functional equations to obtain meromorphic continuation
of the series involved to the the entirety of $\mc^2$, we make a crucial use of the functional equation of a general (not necessarily primitive)
quadratic Hecke $L$-function established in Proposition \ref{Functional equation with Gauss sums} below to convert the original series to a
dual series involving with quadratic Gauss sums. A careful analysis using the ideas of K. Soundararajan and M. P. Young in \cite{S&Y} then
allows us to extend the related series meromorphically to a region that is large enough for our purpose. \newline

  The same approach also allows us to study the smoothed first moment of families of higher order Hecke $L$-functions unconditionally. For
this, we have the following result.
\begin{theorem}
\label{Thmfirstmomentjlarge}
	With the notation as above, let $K=\mq(i)$ or $\mq(\sqrt{-3})$ and let $j=4$ for $K=\mq(i)$, $j=3$ or $6$ for $K=\mq(\sqrt{-3})$.  We have
for $1/2>\Re(\alpha)>0$, all large $X$ and any $\varepsilon>0$,
\begin{align} \label{FirstmomentSmoothedjlarge}
\begin{split}
&  \frac {1}{\#h_{(S_{K,j})}}\sum_{\chi \bmod {S_{K,j}}} \sum_{0 \neq n \in \mathcal O_K} L(\tfrac{1}{2}+\alpha, \chi \cdot \chi^{(n)}_{j})\Phi
\left( \frac {N(n)}X \right) \\
=&  X\widehat \Phi(1) \frac {r_K}{\#h_{(S_{K,j})}}\sum_{\chi \bmod {S_{K,j}}} \frac{L(j(\tfrac{1}{2}+\alpha), \chi^j)}{L(1+j(\tfrac{1}{2}+\alpha),
\chi^j)}+O\lz(1+|\alpha|)^{1+\varepsilon}|\alpha-1/2|^{-1}  X^{1/2+1/j-\alpha+\varepsilon}\pz.
\end{split}
\end{align}
\end{theorem}

The averages of various families of Hecke $L$-functions similar to those considered in Theorem \ref{Thmfirstmomentjlarge} were evaluated in \cites{FaHL, Luo, G&Zhao1, G&Zhao2023-4}. Our treatment here is more direct and yields better error terms. Unlike the quadratic case, the proof of Theorem \ref{Thmfirstmomentjlarge} relies heavily on the work of S. J. Patterson \cite{P} that gives the Gauss sums associated to higher order Hecke characters are small on average.

\section{Preliminaries}
\label{sec 2}

\subsection{Imaginary quadratic number fields}
\label{sect: Kronecker}

   Recall that for an arbitrary number field $K$, we denote $\mathcal{O}_K, U_K$ and $D_K$ for the ring of integers, the group of units and the discriminant of $K$, respectively. For any positive rational integer $n$, we say an element $c \in \mathcal{O}_K$ is $n$-th power free if no $n$-th power of any prime divides $c$. \newline

In the rest of this section, we let $K$ be an imaginary quadratic number field. Unless otherwise mentioned, the following facts concerning an imaginary quadratic number field $K$ can be found in \cite[Section 3.8]{iwakow}. \newline

  The ring of integers $\mathcal{O}_K$ is a free $\mz$ module (see \cite[Section 3.8]{iwakow}) such that $\mathcal{O}_K=\mz+\omega_K \mz$, where
\begin{align*}
   \omega_K & =\begin{cases}
\displaystyle     \frac {1+\sqrt{d}}{2} \qquad & d \equiv 1 \pmod 4, \\ \\
     \sqrt{d} \qquad & d \equiv 2, 3 \pmod 4.
    \end{cases}
\end{align*}
If $K=\mq(\sqrt{-3})$, we also have $\mathcal{O}_K=\mz+\omega \mz$, recalling that $\omega=(-1+\sqrt{-3})/2$. \newline

   The group of units are given by
\begin{align}
\label{Uk}
   U_K & =\begin{cases}
     \{ \pm 1 , \pm i \} \qquad & d=-1, \\
     \{ \pm 1, \pm \omega_K, \pm \omega_K^2 \}=\{ \pm 1, \pm \omega, \pm \omega^2 \} \qquad & d=-3,  \\
     \{ \pm 1  \} \qquad & \text{other $d$}.
    \end{cases}
\end{align}

  Moreover, the discriminants are given by
\begin{align*}
   D_K & =\begin{cases}
     d \qquad & \text{if $d \equiv 1 \pmod 4$ }, \\
     4d \qquad & \text{if $d \equiv 2, 3 \pmod 4$ }.
    \end{cases}
\end{align*}

If we write $n=a+b\omega_K$ with $a, b \in \mz$, then
\begin{align*}
 N(n)=\begin{cases}
   \displaystyle  a^2+ab+b^2\frac {1-d}{4} \qquad & d \equiv 1 \pmod 4,  \\
   \displaystyle  a^2-db^2 \qquad &  d \equiv 2, 3 \pmod 4.
    \end{cases}
\end{align*}

   Denote the Kronecker symbol in $\mz$ by $\leg {\cdot}{\cdot}_{\mz}$.  Then one factors a prime ideal $(p) \in \mz$ in $\mathcal O_K$ in the following way:
\begin{align*}
    & \leg {D_K}{p}_{\mz} =0, \text{then $p$ ramifies}, \ (p)=\mathfrak{p}^2 \quad \text{with} \quad N(\mathfrak{p})=p, \\
    & \leg {D_K}{p}_{\mz} =1, \text{then $p$ splits}, \ (p)=\mathfrak{p}\overline{\mathfrak{p}} \quad \text{with} \quad  \mathfrak{p} \neq \overline{\mathfrak{p}}
    \quad \text{and} \quad  N(\mathfrak{p})=p, \\
    & \leg {D_K}{p}_{\mz}=-1, \text{then $p$ is inert}, \ (p)=\mathfrak{p} \quad \text{with} \quad  N(\mathfrak{p})=p^2.
\end{align*}

\subsection{Residue symbols}
\label{sec2.4}

Let $K^{\times}$ be the multiplicative group $K \backslash \{0 \}$. Recall that we reserve the letter $\varpi$ for a prime in $\mathcal O_K$. Let $j$ be a positive rational integer and $P_K$ be the set of prime ideals $(\varpi) \in \mathcal O_K$ such that $\varpi | j$.
Let $\mu_j(K)=\{ \zeta \in K^{\times}: \zeta^j=1 \}$ and suppose that $\mu_j(K)$ has $j$ elements.  For any $q \in \mathcal{O}_K$, let $\left
(\mathcal{O}_K / (q) \right )^{\times}$ denote the group of reduced residue classes modulo $q$, i.e. the multiplicative group of invertible
elements in $\mathcal{O}_K / (q)$. We now define the $j$-th order residue symbol $\leg {\cdot}{\cdot}_j$ in $\mathcal{O}_K $.  Note that the
discriminant of $x^j-1$ is divisible only by the primes in $P_K$. It follows that for any prime $\varpi \in \mathcal{O}_K,  \varpi \not \in
P_K$, the map
\begin{align*}
    \zeta \mapsto \zeta \pmod \varpi : \mu_j(K) \rightarrow \mu_j(\mathcal{O}_K/(\varpi))=\{ \zeta \in (\mathcal{O}_K/(\varpi))^{\times}:
\zeta^j=1 \}
\end{align*}
   is bijective. For such $\varpi$,  we define for $a \in \mathcal{O}_K$, $(a, \varpi)=1$ by $\leg{a}{\varpi}_j \equiv
a^{(N(\varpi)-1)/j} \pmod{\varpi}$, with $\leg{a}{\varpi}_j \in \mu_j(K)$. When $\varpi | a$, we define $\leg{a}{\varpi}_j =0$.   The $j$-th
order residue symbol is extended to any $c \in \mathcal O_K$ with $(c, j)=1$ multiplicatively. Here we define $\leg {\cdot}{c}_j=1$ when $c$
is a unit in $\mathcal{O}_K$. For any $1 \leq l \leq j$,  we note that $\leg {\cdot}{c}^l_j=$ is a $j/(j,l)$-th order residue symbol. Also, we
note that whenever $\leg {\cdot}{n}_j$ is defined for an element $n \in \mathcal O_K$, we deduce readily from the definition that we have for
any $u \in U_K$,
\begin{align}
\label{2.05}
  \leg {u}{n}_j=u^{(N(n)-1)/j}.
\end{align}

Our discussions above apply in particular to the number field $K=\mq(\sqrt{d})$ with $d$ taking values from \eqref{dvalue}. It follows from our discussions above and \eqref{Uk} that the $j$-th order residue symbol is defined for $j=2$ for all such $d$, for $j=4$ when $d=-1$ and for $j=3, 6$ when $d=-3$. \newline

   It is well-known that when a number field $K$ is of class number one, then every ideal of $\mathcal O_K$ is principal, so that one may fix
a unique generator for each non-zero ideal and we call these generators primary elements. As this is the case for $K=\mq(\sqrt{d})$ with $d$ taking values in \eqref{dvalue}, we now describe our choice of primary elements for these $K$. \newline

  For $K=\mq(i)$, we note that $(1+i)$ is the only ideal above the rational ideal $(2) \in \mq$, and we define $1+i$ to be primary. We also define for any $n \in \mathcal O_K$ with $(n, 2)=1$ to be primary if and only if $n \equiv 1 \pmod {(1+i)^3}$ by observing that the group $\left (\mathcal{O}_K / ((1+i)^3) \right )^{\times}$ is isomorphic to $U_K$.  We then use the above defined primary elements to generator all primary elements in $\mathcal O_K$ multiplicatively. \newline

 For $K=\mq(\sqrt{-3})$, we actually define two sets of primary elements for different purposes. To distinguish them, we shall call the elements in the first set by primary elements and the elements in the second set by $E$-primary elements. For both sets, we define $1-\omega$ to be
primary by noting that $(1-\omega)$ is the only ideal above the rational ideal $(3) \subset \mq$ (see \cite[Proposition 9.1.4]{I&R}). Next we
define for any $n \in \mathcal O_K$ with $(n, 3)=1$ to be primary (hence in the first set) if and only if $n \equiv 1 \pmod {3}$ by
observing that the group $\left (\mathcal{O}_K / (3) \right )^{\times}$ is isomorphic to $U_K$. Now for the second set, we notice that the rational ideal $(2)$ is inert in $\mathcal O_K$ and we define $2$ to be $E$-primary. For the rest, we directly define for any $n=a+b\omega \in \mathcal O_K$ with $a, b \in
\mz$ and $(n, 6)=1$ to be $E$-primary (hence in the second set) if $n \equiv \pm 1 \pmod 3$ and satisfies
\begin{align*}
\begin{split}
   & a+b \equiv 1 \pmod 4, \quad  \text{if} \quad 2 | b,  \\
   & b \equiv 1 \pmod 4, \quad  \text{if} \quad 2 | a,  \\
   & a \equiv 3 \pmod 4, \quad  \text{if} \quad 2 \nmid ab.
\end{split}
\end{align*}
  Our definition of $E$-primary elements follows from  the notations in \cite[Section 7.3]{Lemmermeyer}. It is shown in \cite[Lemma
7.9]{Lemmermeyer} that any $n$ with $(n,6)=1$ is $E$-primary if and only if $n^3=c+d\omega$ with $c, d \in \mz$ such that $6 | d$ and $c+d
\equiv 1 \pmod 4$. This implies that products of $E$-primary elements are again $E$-primary. We now extend the above defined primary or $E$-primary elements to define all primary or $E$-primary  elements in $\mathcal O_K$ multiplicatively.  It is easy to see that each ideal in $\mathcal O_K$ does have a unique primary or $E$-primary generator. \newline

   Now, for other values of $d$ given in \eqref{dvalue}, there is no classical choice for primary elements in the corresponding number field $K=\mq(\sqrt{d})$. This is partially due to the fact that in this case, the group of units $U_K$ does not span the group of any reduced residue classes.  Our definition of primary elements to be described in what follows in this case is to ensure that they (when applicable) satisfy the quadratic reciprocity law given in \eqref{quadreciKprim} below, which in turn makes sure that the Gauss sum $G_K(k,\chi_{2, n})$ defined in \eqref{GKdef} below is multiplicative with respect to $n$. \newline

   We start with the following quadratic reciprocity law concerning any imaginary quadratic number field, a result that is taken from \cite[Theorem 12.17]{Lemmermeyer05}.
\begin{lemma}
\label{Quadrecgen}
   Let $K=\mq(\sqrt{d})$ with $d$ being negative and square-free. For any $\alpha \in \mathcal O_K, (\alpha, 2)=1$, we define elements $t_{\alpha}, t'_{\alpha} \in \{ 0,1 \}$ by
\begin{align}
\label{alphadef}
   \alpha \equiv  \begin{cases}
   \sqrt{d}^{t_{\alpha}} (1+2\sqrt{d})^{t'_{\alpha}}c^2  \pmod 4 \qquad & \text{if} \qquad d \equiv 1 \pmod 2,\\
    (1+\sqrt{d})^{t_{\alpha}} (-1)^{t'_{\alpha}}c^2 \pmod 4 \qquad & \text{if} \qquad d \equiv 0 \pmod 2,
\end{cases}
\end{align}
   for some $c \in \mathcal O_K$. Then the quadratic reciprocity law for co-prime elements $n,m$ with $(nm, 2)=1$ is given by
\begin{align*}
    \leg {n}{m}_2\leg{m}{n}_2=(-1)^{T}, \quad \mbox{where} \quad
   T \equiv  \begin{cases}
   t_{n}t'_{m}+ t'_{n}t_{m}+ t_{n}t_{m}\pmod 2 \qquad & \text{if} \qquad d \equiv 1, 2 \pmod 4,\\
    t_{n}t'_{m}+ t'_{n}t_{m} \pmod 2 \qquad & \text{if} \qquad d \equiv 3 \pmod 4.
\end{cases}
\end{align*}
\end{lemma}

  We apply the above lemma to the case when $d$ takes the values given in \eqref{dvalue} but $d \neq -1, -3$. Note that those values of $d$ satisfy $d \equiv 1, 2 \pmod 4$. When $d \equiv 1 \pmod 4$, we write $d=4k+1$ and summarize from \cite{Lemmermeyer05} the following table for the values of $t_{\alpha}, t'_{\alpha}$:
\begin{align}
\label{tabled1}
   \begin{array}{|c|c|c|c|c|}     \hline
 \alpha \pmod 4,  2 \nmid k & 1, k+\omega_K & -1, -(k+\omega_K) & 1+2\omega_K, -k+\omega_K & -(1+2\omega_K),  -(-k+\omega_K) \\ 
 & k+1-\omega_K & -(k+1-\omega_K) & -k+1+\omega_K & -(-k+1+\omega_K) \\ \hline
  \alpha \pmod 4, 2 | k  & 1 & -1 & 1+2\omega_K & -(1+2\omega_K)  \\ \hline
    N(\alpha) \equiv & 1 \pmod 4 & 1 \pmod 4 & -1 \pmod 4 & -1 \pmod 4 \\
    t_{\alpha} & 0 & 0 & 1 &1 \\
     t'_{\alpha} & 0 & 1 & 1 &0 \\
     \hline
   \end{array}
\end{align}
  Here we point out that one may compute $\alpha \pmod 4$ using the discussions given in \cite[Section 2.2]{G&Zhao2309}. \newline

   Similarly, in the case of $d \equiv 2 \pmod 4$, \cite{Lemmermeyer05} gives the table below:
\begin{align}
\label{tabled2}
   \begin{array}{|c|c|c|c|c|}
    \hline
    \alpha \pmod 4& 1, 3+2\sqrt{d} & -1, -(3+2\sqrt{d}) & \pm 1+\sqrt{d} & -(\pm 1+\sqrt{d}) \\ \hline
    N(\alpha) \equiv & 1 \pmod 4 & 1 \pmod 4 & -1 \pmod 4 & -1 \pmod 4 \\
    t_{\alpha} & 0 & 0 & 1 &1 \\
     t'_{\alpha} & 0 & 1 & 1 &0 \\
     \hline
   \end{array}
\end{align}

   We now give a description of the primary elements of $K=\mq(\sqrt{d})$ with $d \in \mathcal{S} \setminus \{ -1, -3 \}$.  We first observe that if $d \neq -2$, then $d \equiv 1 \pmod 4$.  Also, the residue classes of $\alpha \pmod 4$ listed in Columns $3$ and $5$ of both tables \eqref{tabled1} and \eqref{tabled2} are negatives of those listed in column $2,4$ of the corresponding tables. As we have $U_K=\{ \pm 1\}$  by \eqref{Uk} for all $K=\mq(\sqrt{d})$ when $d$ takes the values given in \eqref{dvalue} but $d \neq -1, -3$, we see from the above tables that any ideal $(\alpha)$ that is co-prime to $2$ has a unique generator that is congruent to one of those listed in Columns $2$ and $4$. We also call these generators primary.  We note from \eqref{alphadef} that the set of primary elements is closed under multiplication. \newline

  Moreover, our discussions in Section \ref{sect: Kronecker} yield that for these $d$'s given in \eqref{dvalue}, $(2)=(\omega_K)^2$ for $d=-2$, $(2)=(\omega_K)({\overline{\omega_K}})$ for $d =-7$, and $(2)$ is inert otherwise. We thus define $\omega_K$ to be primary when $d=-2, -7$ and $2$ to be primary for other values of $d$. We then use the above defined primary elements to generator all primary elements in $\mathcal O_K$ multiplicatively. \newline

   Now, one checks readily from \eqref{tabled1} and \eqref{tabled2} and Lemma \ref{Quadrecgen} that for primary elements $n, m$ with $(nm, 2)=1$,
\begin{align*}
   T \equiv  \frac {N(n)-1}{2}\cdot \frac {N(m)-1}{2} \pmod 2.
\end{align*}

From this and Lemma \ref{Quadrecgen} the following quadratic reciprocity law concerning primary elements, a result of independent interest.
\begin{lemma} \label{Quadrecgenprim}
   Let $K=\mq(\sqrt{d})$ with $d \in \mathcal{S} \setminus \{ -1, -3 \}$. For any co-prime primary elements $n,m$ with $(nm, 2)=1$, we have
\begin{align} \label{quadreciKprim}
    \leg {n}{m}_2\leg{m}{n}_2=(-1)^{(N(n)-1)/2\cdot (N(m)-1)/2}.
\end{align}
\end{lemma}

  In the rest of the section, we write down various reciprocity laws and supplementary laws for $K=\mq(i)$ and $K=\mq(\sqrt{-3})$.
For $K=\mq(i)$, the following quartic reciprocity law (see \cite[Theorem 6.9]{Lemmermeyer}) holds for two odd co-prime primary elements $m, n
\in \mathcal{O}_{K}$,
\begin{align*}
 \leg{m}{n}_4 = \leg{n}{m}_4(-1)^{((N(n)-1)/4)((N(m)-1)/4)}.
\end{align*}

  We also deduce from \cite[Theorem 6.9]{Lemmermeyer} that the following supplementary laws for primary $n=a+bi, (n,2)=1$ with $a, b \in \mz$,
\begin{align*}
  \leg {i}{n}_4=i^{(1-a)/2} \qquad \mbox{and} \qquad  \hspace{0.1in} \leg {1+i}{n}_4=i^{(a-b-1-b^2)/4}.
\end{align*}

    As $\leg{m}{n}^2_4=\leg{m}{n}_2$, it follows from the above discussions that the following quadratic reciprocity law holds for odd co-prime primary elements $m, n \in \mathcal{O}_{K}$,
\begin{align}
\label{quadreciKm1}
    \leg {n}{m}_2\leg{m}{n}_2=1.
\end{align}

   Also, the following supplementary laws hold for primary $n=a+bi$, $(n,2)=1$ with $a, b \in \mz$,
\begin{align*}
  \leg {i}{n}_2=(-1)^{(1-a)/2} \qquad \mbox{and} \qquad  \hspace{0.1in} \leg {1+i}{n}_2=(-1)^{(a-b-1-b^2)/4}.
\end{align*}

For $K=\mq(\sqrt{-3})$, we have the following cubic reciprocity law (see \cite[Theorem 7.8]{Lemmermeyer}) for two co-prime primary
elements $m, n \in \mathcal{O}_{K}$ with $(mn, 3)=1$,
\begin{align} \label{cubicrec}
    \leg {n}{m}_3 =\leg{m}{n}_3.
\end{align}

The following supplementary laws (see \cite[Theorem 7.8]{Lemmermeyer}) for a primary hods for $n=a+b\omega$, $(n,3)=1$ with $a, b \in \mz$,
\begin{align}
\label{cubicsupp}
  \leg {\omega}{n}_3=\omega^{(1-a-b)/3} \qquad \mbox{and} \qquad \leg {1-\omega}{n}_3=\omega^{(a-1)/3}.
\end{align}

Moreover, a sextic reciprocity law holds for two $E$-primary, co-prime numbers $n, m \in \mathcal O_K$ with $(mn, 6)=1$.  Namely,
\begin{align}
\label{quadreciQw}
    \leg {n}{m}_6 =\leg{m}{n}_6(-1)^{((N(n)-1)/2)((N(m)-1)/2)}.
\end{align}

Now \cite[Theorem 7.10]{Lemmermeyer} gives supplementary laws for for an $E$-primary $n=a+b\omega, (n,6)=1$ with $a, b \in
\mz$,
\begin{align}
\label{supp2}
  \leg {1-\omega}{n}_2=\leg {a}{3}_{\mz} \qquad \mbox{and} \qquad  \hspace{0.1in} \leg {2}{n}_2=\leg
  {2}{N(n)}_{\mz},
\end{align}
  where $\leg {\cdot}{\cdot}_{\mz}$ denotes the Kronecker symbol in $\mz$.  As $\leg{m}{n}^3_6=\leg{m}{n}_2,
\leg{m}{n}^2_6=\leg{m}{n}_3$, it readily emerges from \eqref{quadreciQw} that we have the following quadratic reciprocity law for two $E$-primary, co-prime
numbers $n, m \in \mathcal O_K$ with $(mn, 6)=1$,
\begin{align}
\label{quadreci}
    \leg {n}{m}_2 =\leg{m}{n}_2(-1)^{((N(n)-1)/2)((N(m)-1)/2)}.
\end{align}

   We also deduce from \eqref{2.05}, \eqref{cubicrec}, \eqref{cubicsupp} and \eqref{supp2} the supplementary laws such that if $n=a+b\omega$
with $a, b \in \mz$ is both primary and $E$-primary,
\begin{align*}
  \leg {1-\omega}{n}_6=\leg {1-\omega}{n}^{3-2}_6=\overline{\omega}^{(a-1)/3}\leg {a}{3}_{\mz},  \qquad \mbox{and} \qquad  \hspace{0.1in} \leg
{2}{n}_6=\leg {2}{n}^{3-2}_6=\overline{\leg
  {n}{2}}_{3} \leg  {2}{N(n)}_{\mz}.
\end{align*}

\subsection{Gauss sums}
\label{section:Gauss}

  Let $K$ be an imaginary quadratic number field.  Let $\mathfrak{m}$ be a non-zero integral ideal of $\mathcal O_K$ and $\chi$ a homomorphism:
\begin{align*}
  \chi: \left (\mathcal{O}_K / \mathfrak{m} \right )^{\times}  \rightarrow S^1 :=\{ z \in \mc :  |z|=1 \}.
\end{align*}
Following the nomenclature of \cite[Section 3.8]{iwakow}, we shall refer $\chi$ as a Dirichlet character modulo $\mathfrak{m}$. When $\mathfrak{m}=(q)$, we also say that $\chi$ is a Dirichlet character modulo $q$. We say that such a Dirichlet character $\chi$ modulo $q$ is primitive if it does not factor through $\left (\mathcal{O}_K / (q') \right )^{\times}$ for any divisor $q'$ of $q$ with $N(q')<N(q)$. \newline

    For any complex number $z$, set $e(z) = \exp (2 \pi i z) = e^{2 \pi i z}$ and $\delta_K=\sqrt{D_K}$ so that the different of $K$ is the principal ideal $(\delta_K)$.  We further define
\begin{align*}
   \widetilde{e}_K(z) =\exp \left( 2 \pi i {\rm Tr} \left( \frac {z}{\delta_K} \right) \right)=\exp \left( 2\pi i  \left( \frac {z}{\sqrt{D_K}} -
\frac {\overline{z}}{\sqrt{D_K}} \right) \right).
\end{align*}

    Now, let $K=\mq(\sqrt{d})$ with $d$ taking values in \eqref{dvalue}. Recall that we write $\chi_{j, n} =\leg {\cdot}{n}_j$ whenever it is defined. For any $k \in
\mathcal{O}_K$, the associated Gauss sum $g_K(k, \chi_{j, n})$ is defined by
\begin{align*}
 g_K(k,\chi_{j, n}) = \sum_{x \shortmod{n}} \chi_{j,n}(x) \widetilde{e}_K\leg{kx}{n}.
\end{align*}

Furthermore, write $g(\chi_{j,n})$ for $g(1, \chi_{j,n})$ and the following result evaluates $g_K(\chi_{2,\varpi})$  for a primary prime $\varpi$.
\begin{lemma}
\label{Gausssum}
    Let $K=\mq(\sqrt{d})$ with $d$ taking values in \eqref{dvalue} and $\varpi$ a primary prime in $\mathcal{O}_K$ such that $(\varpi, B_K)=1$. Then
\begin{align}
\label{gvalue}
   g_K(\chi_{2,\varpi})= \begin{cases}
    N(\varpi)^{1/2} \qquad & \text{for all $d$, if} \ N(\varpi) \equiv 1 \pmod 4 \\
    -i N(\varpi)^{1/2} \qquad & \text{for $d \neq -7$, if} \  N(\varpi) \equiv -1 \pmod 4,\\
     i N(\varpi)^{1/2} \qquad & \text{for $d = -7$, if} \ N(\varpi) \equiv -1 \pmod 4.
\end{cases}
\end{align}
\end{lemma}
\begin{proof}
  Note first that \eqref{gvalue} has been established for the case $d=-1$ and $d=-3$ in \cite[Lemma 2.2]{G&Zhao4} and \cite[Lemma 2.3]{G&Zhao3}, respectively. For other values of $d$, we observe that as $U_K=\{ \pm 1\}$  by \eqref{Uk} for the corresponding number field $K$, we have either $\varpi$ or $-\varpi$ is primary for any prime $\varpi$. Moreover, by a change of variable $x \mapsto -x$,
\begin{align}
\label{gpirel}
   g_K(\chi_{2,\varpi})= \leg {-1}{\varpi}_2g_K(\chi_{2,-\varpi}).
\end{align}
 It follows from this and \eqref{2.05} that $g_K(\chi_{2,\varpi})= g_K(\chi_{2,-\varpi})$ when $N(\varpi) \equiv 1 \pmod 4$. Thus the value of $g_K(\chi_{2,\varpi})$ is fixed regardless of whether $\varpi$ is primary or not.  We then deduce from this and \cite[Lemma 2.1]{G&Zhao2022-4} that the assertion of the lemma is valid when $N(\varpi) \equiv 1 \pmod 4$. \newline

In the case $N(\varpi) \equiv -1 \pmod 4$, we write $\varpi=a+b\omega_K$ with $a, b \in \mz$. We deduce from Table \eqref{tabled1} that when $d \equiv 1 \pmod 4, d=4k+1$, $\varpi$ is primary implies that $b \equiv 1 \pmod 4$ or $2|b, a+b(1-d)/4 \equiv -1 \pmod 4$ in the case $2 \nmid k$, while $\varpi$ is primary implies that $2|b, a+b(1-d)/4 \equiv 1 \pmod 4$ in the case $2|k$. Similarly, from Table \eqref{tabled2} that when $d =-2$, $\varpi$ is primary if and only if $b \equiv 1 \pmod 4$. In all these above cases, the notation of primary defined in Section \ref{sec2.4} is consistent with that defined in \cite[Section 2.3]{G&Zhao2022-4} save for the case $d=-7$. We thus deduce from \cite[Lemma 2.1]{G&Zhao2022-4} directly that the assertion of the lemma is valid for the middle entry in \eqref{gvalue}. \newline

Finally, if $d=-7$, for any primary $\varpi$ with $N(\varpi) \equiv -1 \pmod 4$, our discussions above imply that $-\varpi$ is a ``primary" element in the sense defined in \cite[Section 2.3]{G&Zhao2022-4}. It thus follows from \cite[Lemma 2.1]{G&Zhao2022-4} that we have $g_K(\chi_{2,-\varpi})= -i N(\varpi)^{1/2}$ in this case. Now applying the relation \eqref{gpirel} and noting, by \eqref{2.05}, that $\leg {-1}{\varpi}_2=-1$ since $N(\varpi) \equiv -1 \pmod 4$ here. This readily allows us to establish the last entry in \eqref{gvalue} and hence completes the proof of the lemma.
\end{proof}

  For $j=2$, we further define for any $n,k \in \mathcal{O}_K, (n,2)=1$,
\begin{align}
\label{GKdef}
  G_K(k,\chi_{2, n}) = \left (\frac {1-i}{2}+\leg {-1}{n}_2\frac {1+i}{2}
 \right )g_K(k,\chi_{2, n}).
\end{align}

The definition of $G_K(k,\chi_{2, n})$ is motivated by the work of K. Soundararajan \cite[Section 2.2]{sound1} and the reason for its introduction is that $G_K(k,\chi_{2, n})$ is now a multiplicative function of $n$.  Note that $G_K(k,\chi_{2, n}) =g_K(k,\chi_{2, n})$ when $K=\mq(i)$ since $\leg {-1}{n}_2=1$ in this case. We now set
\begin{align}
\label{etadef}
 \eta_K & =\begin{cases}
    i, \qquad & d=-1,\\
    1, \qquad & d=-7,\\
    -1, \qquad & \text{otherwise}.
\end{cases}
\end{align}
  We further denote $\varphi_K(n)$ for the analogue of the classical Euler totient
function in $\mathcal O_K$ so that it is a multiplicative function of $n$ and that
$\varphi_K(\varpi^l)=\#(\mathcal{O}_K/(\varpi^l))^{\times}=N(\varpi^l)-N(\varpi^{l-1})$ for every prime $\varpi$ and every $l \in \natn$.  The following result gives the properties of $G_K(k,\chi_{2, n})$.
\begin{lemma} \label{GausssumMult}
  With the notation as above and let $K=\mq(\sqrt{d})$  with $d$ taking values in \eqref{dvalue}, the following identities hold.

\begin{enumerate} [(i)]
\item \label{item1} If $(m,n)=1$ and $(mn, 2)=1$, then we have $G_K(k,\chi_{2, mn})=G_K(k,\chi_{2, m})G_K(k,\chi_{2, n})$.

\item \label{item2}For $(n,2)=1$, we have
\begin{align*}
   G_K(rs,\chi_{2, n}) & = \overline{\leg{s}{n}}_2 G_K(r,\chi_{2, n}), \qquad (s,n)=1.
\end{align*}
\item \label{item3} Let $\varpi$ be a primary prime in $\mathcal{O}_K$ that is co-prime to $2$ (resp. $6$) for $K=\mq(i)$ (resp. $K=\mq(\sqrt{-3})$).
Suppose $\varpi^{h}$ is the largest power of $\varpi$ dividing $k$. (If $k = 0$ then set $h =
\infty$.) Then for $l \geq 1$,
\begin{align*}
 G_K(k, \chi_{2, \varpi^l})& =\begin{cases}
    0 \qquad & \text{if} \qquad l \leq h \qquad \text{is odd},\\
    \varphi_K(\varpi^l)=\#(\mathcal{O}_K/(\varpi^l))^{\times} \qquad & \text{if} \qquad l \leq h \qquad \text{is even},\\
    -N(\varpi)^{l-1} & \text{if} \qquad l= h+1 \qquad \text{is even},\\
    \leg {\eta_K k\varpi^{-h}}{\varpi}_2 N(\varpi)^{l-1/2} \qquad & \text{if} \qquad l= h+1 \qquad \text{is odd},\\
    0 \qquad & \text{if} \qquad l \geq h+2,
\end{cases}
\end{align*}
  where $\eta_K$ is given in \eqref{etadef}.
\end{enumerate}
\end{lemma}
\begin{proof}
    The statement of \eqref{item1} can be established by the arguments given in the proof of \cite[Lemma 2.3]{sound1}, upon using the quadratic reciprocity laws \eqref{quadreciKprim},  \eqref{quadreciKm1} and \eqref{quadreci}. The cases \eqref{item2} and \eqref{item3} for $K=\mq(i)$ is
given in \cite[Lemma 2.2]{G&Zhao4}. These cases for other $K$'s are similarly proved, upon using Lemma \ref{Gausssum}.
 This completes the proof of the lemma.
\end{proof}

\subsection{Hecke $L$-functions}
\label{sect: Lfcn}

  Let $K$ be an arbitrary imaginary quadratic number field. For any $l \in \mz$, we define a unitary character
$\chi_{\infty}$ from $\mc^*$ to $S^1$ by:
\begin{align*}
  \chi_{\infty}(z)=\leg {z}{|z|}^l.
\end{align*}
  The integer $l$ is called the frequency of $\chi_{\infty}$. \newline

  Now, given a Dirichlet character $\chi$ modulo $\mathfrak{m}$ for a non-zero integral ideal $\mathfrak{m}$ and a unitary character $\chi_{\infty}$ such that for any $n \in U_K$,
\begin{align*}
  \chi(n)\chi_{\infty}(n)=1,
\end{align*}
we can define a Hecke character $\psi$ modulo $\mathfrak{m}$ (see \cite[\S 6, Chap. VII]{Neukirch}) on $I_{\mathfrak{m}}$ in $K$ in the
way that for any $(\alpha) \in I_{\mathfrak{m}}$,
\begin{align*}
  \psi((\alpha))=\chi(\alpha)\chi_{\infty}(\alpha).
\end{align*}
If $l=0$, we say that $\psi$ is a Hecke character modulo $q$ of trivial infinite type. In this case, the Hecke character $\psi$ is trivial on
$U_K$ so that we may regard $\psi$ as defined on $\mathcal{O}_K$ instead of on $I_{\mathfrak{m}}$, setting $\psi(\alpha)=\psi((\alpha))$ for any $\alpha \in \mathcal{O}_K$. We may also write $\chi$ for $\psi$ as well, since $\psi(\alpha)=\chi(\alpha)$ for any $\alpha \in \mathcal{O}_K$. When $\mathfrak{m}=(q)$, we say such a Hecke character $\chi$ is primitive modulo $q$ if $\chi$ is a primitive Dirichlet character. Likewise, we say that  $\chi$
is induced by a primitive Hecke character $\chi'$ modulo $q'$ if $\chi(\alpha)=\chi'(\alpha)$ for all $(\alpha, q')=1$. \newline

 Note that the symbol $\chi^{(q)}_{j}$ is trivial on $U_K$ whenever it is defined, hence can be regarded as a Hecke character of trivial
infinite type. Similarly, any ray class group character on $h_{(q)}$ can then be viewed as a Hecke character modulo $q$ of trivial
infinite type. \newline

 For any Hecke character $\chi$, the associated Hecke $L$-function is defined for $\Re(s) > 1$ to be
\begin{equation*}
  L(s, \chi) = \sum_{0 \neq \mathcal{A} \subset
  \mathcal{O}_K}\chi(\mathcal{A})(N(\mathcal{A}))^{-s},
\end{equation*}
where $\mathcal{A}$ runs over all non-zero integral ideals in $K$ and $N(\mathcal{A})$ is the norm of $\mathcal{A}$. \newline

If $\chi$ is a primitive Hecke character of trivial infinite type modulo $n$, we extend the definition of the Gauss sum introduced in the
previous section to the associated Gauss sum $g_K(k, \chi)$ for any $k \in \mathcal{O}_K$ by
\begin{align*}
 g_K(k,\chi) = \sum_{x \shortmod{n}} \chi(x) \widetilde{e}_K\leg{kx}{n}.
\end{align*}
Note as $\chi$ is trivial on $U_K$, the above definition is independent of the choice of the generator for the ideal $(n)$. \newline

In this case, a well-known result of E. Hecke gives that $L(s, \chi)$ has analytic continuation to the whole complex plane and satisfies the functional equation (see \cite[Theorem 3.8]{iwakow})
\begin{align}
\label{fneqn}
  \Lambda(s, \chi) = W(\chi)\Lambda(1-s, \overline\chi),
\end{align}
 where
\begin{align} \label{Lambda}
  W(\chi) = g_K(1,\chi)(N(q))^{-1/2} \quad \mbox{and} \quad \Lambda(s, \chi) = (|D_K|N(q))^{s/2}(2\pi)^{-s}\Gamma(s)L(s, \chi).
\end{align}
  Note also that $|W(\chi)|=1$.  It follows from \eqref{fneqn} and\eqref{Lambda} that
\begin{align}
\label{fneqnL}
  L(s, \chi)=W(\chi)(|D_K|N(q))^{1/2-s}(2\pi)^{2s-1}\frac {\Gamma(1-s)}{\Gamma (s)}L(1-s, \overline\chi).
\end{align}

\subsection{Functional Equation of Hecke $L$-functions}
    In this section, we develop a functional equation for Hecke $L$-functions associated to a general (not necessary primitive) Hecke character of trivial infinite type for an imaginary quadratic number field $K$ of class number one. We first note the following Poisson summation formula (established in the proof of
\cite[Lemma 2.5]{G&Zhao2022-4}) for smoothed character sums over all elements in $\mathcal{O}_K$.
\begin{lemma}
\label{Poissonsum} Let $K$ be an arbitrary imaginary quadratic number field and let $\chi$ be a Hecke character of trivial infinite type modulo $n$. For an Schwartz class function $W:\mr^{+} \rightarrow \mr$,  we have for $X>0$,
\begin{align}
\label{PoissonsumQw}
   \sum_{m \in \mathcal{O}_K}\chi(m)W\left(\frac {N(m)}{X}\right)=\frac {X}{N(n)}\sum_{k \in
   \mathcal{O}_K}g_K(k,\chi)\widetilde{W}_K\left(\sqrt{\frac {N(k)X}{N(n)}}\right),
\end{align}
   where
\begin{align*}
   \widetilde{W}_K(t) &=\int\limits^{\infty}_{-\infty}\int\limits^{\infty}_{-\infty}W(N(x+y\omega_K))\widetilde{e}\left(-
t(x+y\omega_K)\right)\dif x \dif y, \quad t
   \geq 0.
\end{align*}
\end{lemma}

   Next, we note that it is given on \cite[p. 184]{G&Zhao2022-4} that
\begin{align*}
     \widetilde{W}_{K}(t)  =
\begin{cases}
\ \displaystyle{   \int\limits_{\mr^2}W \left( x^2+xy+y^2\frac {1-d}{4} \right)e(ty) \ \dif x \dif y}  \qquad & \text{if $d \equiv 1 \pmod
4$}, \\ \\
\ \displaystyle{    \int\limits_{\mr^2}W \left( x^2-dy^2 \right) e(ty) \ \dif x \dif y} \qquad & \text{if $d \equiv 2, 3 \pmod 4$}.
\end{cases}
\end{align*}

   We make a change of variables in the above integrals by setting
\[  x+\frac {y}{2}  =x', \; \frac {\sqrt{-d}}{2}y=y',  \; \text{if $d \equiv 1 \pmod 4$} \quad \text{and} \quad x=x', \; \sqrt{-d}y=y', \; 
\text{if $d \equiv 2, 3 \pmod 4$}. \]

The Jacobian of these transformations are $\sqrt{-d}/2$ for $d \equiv 1 \pmod 4$ and $\sqrt{-d}$, otherwise. This leads to
\begin{align*}
     \widetilde{W}_{K}(t)  =
\displaystyle{   \frac {2}{\sqrt{|D_K|}}\int\limits_{\mr^2}W \left( x^2+y^2 \right) e \left( \frac {2ty}{\sqrt{|D_K|}} \right) \ \dif x \dif y
} .
\end{align*}

 Applying the above expression with $W(s)=e^{-s}$, we see that
\begin{align*}
   \widetilde{W}_K(t) =& \frac {2}{\sqrt{|D_K|}}\int\limits^{\infty}_{-\infty}\int\limits^{\infty}_{-\infty} \exp \left( -(x^2+y^2)-\frac {4\pi ty i
}{\sqrt{|D_K|}} \right) \dif x \dif y =\frac {2\pi}{\sqrt{|D_K|}} \exp \left( -\frac {4\pi^2t^2}{|D_K|} \right).
\end{align*}
Now it follows from this and \eqref{PoissonsumQw} that we have
\begin{align*}
   \sum_{m \in \mathcal{O}_K}\chi(m)\exp \left(-2\pi y N(m)\right)=\frac {1}{\sqrt{|D_K|} y N(n)}\sum_{k \in
   \mathcal{O}_K}g_K(k,\chi)\exp \left(-\frac {2\pi N(k)}{|D_K| y N(n)}\right).
\end{align*}

 Suppose now that $K$ has class number one and $\chi$ satisfies $\chi(0)=0, g_K(0,\chi)=0$, we deduce
that
\begin{align}
\label{PoissonsumWet1}
   \sum_{m \in \mathcal{O}_K}\chi(m)\exp \left(-2\pi y N(m)\right)=\frac {1}{\sqrt{|D_K|} y N(n)}\sum_{\substack{ k \neq 0 \\ k \in
   \mathcal{O}_K}}g_K(k,\chi)\exp \left(-\frac {2\pi N(k)}{|D_K| y N(n)}\right).
\end{align}

   We consider the Mellin transform of the left-hand side above.  This leads to
\begin{equation}
\label{keyleft}
	\int\limits_{0}^{\infty}y^{s}\sum_{\substack{m \neq 0 \\ m \in \mathcal{O}_K}}\chi(m)\exp \left(-2\pi y N(m)\right)\Big )\frac{\dif
y}{y}=\sum_{\substack{m \neq 0 \\ m \in \mathcal{O}_K}}\chi(m)\int\limits_{0}^{\infty}y^{s}e^{-2\pi N(m)y}\frac{\dif
y}{y}=|U_K|(2\pi)^{-s}\Gamma(s) L(s,\chi).
\end{equation}

   On the other hand, the Mellin transform of the right-hand side of \eqref{PoissonsumWet1} equals
\begin{align} \label{keyright}
 \frac {1}{\sqrt{|D_K|} N(n)}  \sum_{\substack{ k \neq 0 \\ k \in
   \mathcal{O}_K}}g_K(k, \chi) \int\limits_{0}^{\infty}y^{s-1}\exp \left(-\frac {2\pi N(k)}{|D_K| y N(n)}\right)\frac{\dif y}{y}.
\end{align}
The change of variable $u=\frac {2\pi N(k)}{|D_K| y N(n)}$ recasts the integral above as
\begin{align}
\begin{split}
\label{integaleval}
 \Big ( \frac {2\pi N(k)}{|D_K|  N(n)} \Big )^{s-1} \int\limits_{0}^{\infty}u^{1-s} e^{-u} \frac{\dif u}{u} =\Big ( \frac
{2\pi N(k)}{|D_K|  N(n)} \Big )^{s-1}\Gamma(1-s).
\end{split}
\end{align}  	

   We deduce from  \eqref{keyleft}--\eqref{integaleval} that
\begin{align*} \begin{split}
  & |U_K|(2\pi)^{-s}\Gamma(s) L(s,\chi)=\frac {1}{\sqrt{|D_K|} N(n)} \Big ( \frac {2\pi }{|D_K|  N(n)} \Big
)^{s-1}\Gamma(1-s)\sum_{\substack{ k \neq 0 \\ k \in
   \mathcal{O}_K}}\frac {g_K(k, \chi)}{N(k)^{1-s}}.
\end{split}
\end{align*}

The following result summarizes our discussions above.
\begin{proposition}
\label{Functional equation with Gauss sums}
   Let $K$ be an imaginary quadratic number field of class number one. For any Hecke character $\chi$ of trivial infinite type modulo $n$ such that $\chi(0)=0$ and $g_K(0,\chi)=0$, we have
\begin{align}
\begin{split}
\label{fcneqnallchi}
  &  L(s, \chi)= N(n)^{-s} \Big(\frac {2\pi}{\sqrt{|D_K|}}\Big )^{2s-1}\frac {\Gamma(1-s)}{|U_K|\Gamma(s)} \sum_{\substack{ k \neq 0 \\ k \in
   \mathcal{O}_K}}\frac {g_K(k, \chi)}{N(k)^{1-s}}.
\end{split}
\end{align}
\end{proposition}

  Note that if $\chi$ is primitive, an argument similar to that given in \cite[\S 9]{Da} shows that $g_K(k, \chi)=g_K(1,\chi)\overline{\chi}(k)$. Thus the functional equation given in \eqref{fcneqnallchi} is in agreement with that given in \eqref{fneqnL}.

\subsection{Bounding $L$-functions}
\label{sec: Lbound}

 Let $K$ be a number field and $\chi$ a Hecke character of trivial infinite type modulo $q$.  Let $\widehat{\chi}$ denote the primitive Hecke character modulo $d$ that induces $\chi$.  Further write $q=q_1q_2$ uniquely such that $(q_1, d)=1$ and that $\varpi |q_2 \Rightarrow \varpi|d$. Then we have
\begin{align}
\label{Ldecomp}
\begin{split}
	L(s,  \chi )=&	\prod_{\varpi | q_1}(1-\widehat{\chi}(\varpi)N(\varpi)^{-s}) \cdot
L(s, \widehat{\chi}).
\end{split}
\end{align}

  Observe that
\begin{align*}
 \Big |1-\widehat{\chi}(\varpi)N(\varpi)^{-s}\Big | \leq 2N(\varpi)^{\max (0,-\Re(s))}.
\end{align*}

Hence
\begin{align}
\label{Lnbound}
\begin{split}
 \Big | \prod_{\varpi | q_1}\Big(1-\widehat{\chi}(\varpi)N(\varpi)^{-s} \Big )\Big | \ll
2^{\mathcal{W}_K(q_1)}N(q_1)^{\max (0,-\Re(s))} \ll N(q_1)^{\max (0,-\Re(s))+\varepsilon},
\end{split}
\end{align}
  where $\mathcal{W}_K$ denotes the number of distinct prime factors of $n$ and the last estimation above follows from the well-known bound
(which can be derived similar to the proof of the classical case over $\mq$ given in\cite[Theorem 2.10]{MVa1})
\begin{align*}
   \mathcal{W}_K(h) \ll \frac {\log N(h)}{\log \log N(h)}, \quad \mbox{for} \quad N(h) \geq 3.
\end{align*}

 Now, the convexity bound for $L(s, \widehat \chi)$ (see \cite[Exercise 3, p. 100]{iwakow}) asserts that
\begin{equation} \label{Lconvexbound}
L( s,  \widehat \chi) \ll
\begin{cases}
\left( N(d)(1+|s|^2) \right)^{(1-\Re(s))/2+\varepsilon}, & 0 \leq \Re(s) \leq 1, \\ 1, & \Re(s)>1.
\end{cases}
\end{equation}

  To estimate $L( s, \widehat \chi)$ for $\Re(s) < 0$, we observe that  Stirling's formula (\cite[(5.113)]{iwakow}) implies that, for
constants $a_0$, $b_0$,
\begin{align}
\label{Stirlingratio}
  \frac {\Gamma(a_0(1-s)+ b_0)}{\Gamma (a_0s+ b_0)} \ll (1+|s|)^{a_0(1-2\Re (s))}.
\end{align}

Thus \eqref{Stirlingratio} and the functional equation \eqref{fneqnL} give that for $\Re(s)<1/2$,
\begin{align}
\label{fneqnquad1}
  L(s, \widehat \chi) \ll (N(d)(1+|s|^2))^{1/2-\Re(s)+\varepsilon}.
\end{align}

We conclude from \eqref{Lconvexbound}--\eqref{fneqnquad1} that
\begin{align} \label{Lchidbound}
\begin{split}
   L(s, \widehat \chi) \ll \begin{cases}
   1,  & \Re(s) >1,\\
   (N(d)(1+|s|^2))^{(1-\Re(s))/2+\varepsilon},  & 0\leq \Re(s) <1,\\
    (N(d)(1+|s|^2))^{1/2-\Re(s)+\varepsilon}, & \Re(s) < 0.
\end{cases}
\end{split}
\end{align}

  Combining \eqref{Ldecomp}, \eqref{Lnbound} and \eqref{Lchidbound} leads to
\begin{align}
\label{Lchidgeneralbound}
\begin{split}
   L(s, \chi) \ll N(q)^{\max (0,-\Re(s))+\varepsilon} (N(q)(1+|s|^2))^{\max \{(1-\Re(s))/2, 1/2-\Re(s), 0\}+\varepsilon}.
\end{split}
\end{align}

Next, note that \cite[Theorem 5.19]{iwakow} yields that when $\Re(s) \geq 1/2+\varepsilon$, under GRH,
\begin{align} \label{PgLest1}
\begin{split}
& \big| L( s,  \widehat\chi )\big |^{-1} \ll |sN(d)|^{\varepsilon}.
\end{split}
\end{align}

  Similar to \eqref{Lnbound}, we also have
\begin{align}
\label{Lninvbound}
\begin{split}
 \Big | \prod_{\varpi | q_1}\Big(1-\widehat{\chi}(\varpi)N(\varpi)^{-s} \Big )^{-1} \Big | \ll
N(q_1)^{\varepsilon}, \quad \mbox{for} \quad \Re(s)>0.
\end{split}
\end{align}

Thus from \eqref{Ldecomp}, \eqref{PgLest1} and \eqref{Lninvbound}, we have, under GRH, for $\Re(s) \geq 1/2+\varepsilon$,
\begin{align}
\label{PgLest2}
\begin{split}
& \big| L( s,  \chi ) \big |^{-1} \ll |sN(q)|^{\varepsilon}.
\end{split}
\end{align}

We shall need the following large sieve result for the $j$-th order Hecke $L$-functions.
\begin{lemma} \label{lem:2.3}
We use the same notation as above and suppose $K=\mq(i)$ or $\mq(\sqrt{-3})$. Let $\chi$ be any fixed Hecke character of trivial infinite type modulo $q$ with $N(q) \leq M$ and $S_j(X)$ denote the set of $j$-th order Hecke characters $\psi$ of trivial infinite type with
conductor not exceeding $X$ when it is non-empty.  We have for $s\in \comc$ and any $\varepsilon>0$ with $\Re(s) \geq 1/2$, $|s-1|>\varepsilon$,
\begin{align}
\label{L1estimation}
\sum_{\substack{\psi \in S_j(X)}} |L(s, \chi \cdot \psi)|
\ll_M & X^{1+\varepsilon} |s|^{1/2+\varepsilon}.
\end{align}
\end{lemma}
\begin{proof}
  For $j>2$, we deduce from the proof of \cite[Corollary 1.4]{BGL} that
\begin{align*}
\sum_{\substack{\psi \in S_j(X)}} |L(s, \chi \cdot \psi)|^2
\ll_M & (X|s|)^{1+\varepsilon}.
\end{align*}
 It is shown in \cite[(4.1)]{G&Zhao3} that the above estimation also holds for $j=2$. The lemma now follows from the above and the Cauchy-Schwarz inequality.
\end{proof}

   For $K=\mq(i)$ or $\mq(\sqrt{-3})$, let $j \geq 1$ be a rational integer such that the $j$-th order residue symbol is defined over $K$.
For any Hecke character $\chi$ of trivial infinite type, the norm of whose modulus is bounded and that $\chi(m) = 0$ at elements $m \in
\mathcal O_K$ whenever $\chi_{j,m}$ is undefined, we define
\begin{align} \label{h}
   h_j(r,s;\chi)=\sum_{\substack{ m \in \mathcal O_K \\ m \text{ primary} }}\frac {\chi(m)g_{K}(r,\chi_{j,m})}{N(m)^s}.
\end{align}

    The following lemma gives the analytic behavior of $h_j(r,s;\chi)$.
 \begin{lemma}
\label{lem1} With the notation as above and $j \geq 3$, the function $h_j(r,s;\chi)$ has meromorphic continuation to the whole complex plane. It
is holomorphic in the region $\sigma=\Re(s) > 1$ except possibly for a pole at $s = 1+1/j$. For any $\varepsilon>0$, with $\sigma_1 = 3/2+\varepsilon$,
then for $\sigma_1 \geq \sigma \geq \sigma_1- 1/2$, we have
\begin{align}
\label{hbound}
  |((j(s-1))^2-1)h_j(r,s;\chi)| \ll N(r)^{(\sigma_1-\sigma+\varepsilon)/2}(1+|s|^2)^{(\sigma_1-\sigma+\varepsilon)(j-1)/2}.
\end{align}
  For $\sigma >\sigma_1$, we have
\begin{equation*}
  |h_j(r,s;\chi)| \ll 1.
\end{equation*}
\end{lemma}

The above result is essentially contained in the Lemma on Page 200 of \cite{P}, except that the estimate \eqref{hbound} is stated only for $|s-(1+1/j)|>1/(2j)$.  An inspection of the proof there  (see the second display on Page 205 of \cite{P}) shows that the estimation given in \eqref{hbound} is in fact valid.

\subsection{Some results on multivariable complex functions}
	
We gather here some results from multivariable complex analysis. We begin with the notion of a tube domain.
\begin{defin}
		An open set $T\subset\mc^n$ is a tube if there is an open set $U\subset\mr^n$ such that $T=\{z\in\mc^n:\ \Re(z)\in U\}.$
\end{defin}
	
   For a set $U\subset\mr^n$, we define $T(U)=U+i\mr^n\subset \mc^n$.  We quote the following Bochner's Tube Theorem \cite{Boc}.
\begin{theorem}
\label{Bochner}
		Let $U\subset\mr^n$ be a connected open set and $f(z)$ a function holomorphic on $T(U)$. Then $f(z)$ has a holomorphic continuation to
the convex hull of $T(U)$.
\end{theorem}

 We denote the convex hull of an open set $T\subset\mc^n$ by $\widehat T$.  Our next result is \cite[Proposition C.5]{Cech1} on the modulus of
holomorphic continuations of multivariable complex functions.
\begin{prop}
\label{Extending inequalities}
		Assume that $T\subset \mc^n$ is a tube domain, $g,h:T\rightarrow \mc$ are holomorphic functions, and let $\tilde g,\tilde h$ be their
holomorphic continuations to $\widehat T$. If  $|g(z)|\leq |h(z)|$ for all $z\in T$ and $h(z)$ is nonzero in $T$, then also $|\tilde g(z)|\leq
|\tilde h(z)|$ for all $z\in \widehat T$.
\end{prop}

\section{Proof of Theorem \ref{Thmfirstmoment} and \ref{Thmfirstmomentjlarge}}
\label{sec Poisson}

\subsection{Initial Treatment}

  Recall that for the number field $K$ considered here, every ideal is generated by a primary element. It follows that for any Hecke character
$\chi$ of trivial infinite type, the associated $L$-function $L(s,\chi)$ is given by a Dirichlet series that for $\Re(s)$ large enough,
\begin{align}
\label{Lseries}
    L(s, \chi) =\sum_{n \odd}\frac {\chi(n)}{N(n)^s}.
\end{align}
   In what follows, we adapt the convention that when $K=\mq(\sqrt{-3})$ and $\chi=\chi_{j, n}$ for $j=2$ or $6$ whenever it is defined, by
primary we always mean $E$-primary. \newline

  We now define functions $A_j(s,w)$ for $j=2,3, 4$ or $6$ for $\Re(s)$, $\Re(w)$ sufficiently large by the absolutely convergent double Dirichlet
series
\begin{align}
\label{Aswexp}
\begin{split}
A_j(s,w)=&
\begin{cases}
\displaystyle \frac {1}{2}\sum_{\chi \in C_K}\sum_{0 \neq n \in \mathcal O_K}\frac{L(w, \chi \cdot \chi^{(n)}_{j})}{N(n)^s}, & j=2, \\
\displaystyle \frac {1}{\#h_{(S_{K,j})}}\sum_{\chi \bmod {S_{K,j}}}\sum_{0 \neq n \in \mathcal O_K}\frac{L(w, \chi \cdot
\chi^{(n)}_{j})}{N(n)^s}, & j >2.
\end{cases}
\end{split}
\end{align}

  When $\chi_{2, m}$ is trivial on $U_K$ for some $m \in \mathcal O_K$, we shall write $\chi_{2, m}(U_K)=1$ to indicate this.  Using the representation given in \eqref{Lseries}, we see that
\begin{align*}
\begin{split}
A_2(s,w)=& \frac {1}{2}\sum_{\chi \in C_K}\sum_{m\odd}\sum_{0 \neq n \in \mathcal O_K}\frac{\chi(m)\leg {n}m_2}{N(n)^sN(m)^w} = \sum_{\substack{ m\odd \\ (m, B_K)=1 \\ \chi_{2,m}(U_K)=1} }\sum_{0 \neq n \in \mathcal O_K}\frac{\leg {n}m_2}{N(n)^sN(m)^w} \\
=& |U_K|\sum_{\substack{ m\odd \\ (m,B_K)=1 \\ \chi_{2,m}(U_K)=1} }\sum_{n\odd}\frac{\leg {n}m_2}{N(n)^sN(m)^w} = |U_K| \sum_{\substack{ m\odd \\ (m,B_K)=1 \\ \chi_{2,m}(U_K)=1} } \frac{L(s, \chi_{2,m})}{N(m)^w}.
\end{split}
\end{align*}
  Similarly, we have for $j>2$,
\begin{align}
\label{Aswexp1}
\begin{split}
A_j(s,w)=& \frac {1}{\#h_{(S_{K,j})}}\sum_{\chi \bmod {S_{K,j}}}\sum_{m\odd}\sum_{0 \neq n \in \mathcal O_K}\frac{\chi(m)\leg
{n}m_j}{N(n)^sN(m)^w} = \frac {1}{\#h_{(S_{K,j})}}\sum_{\substack{ m\odd \\ m \equiv 1 \bmod {S_{K,j}} } }\sum_{0 \neq n \in \mathcal O_K}\frac{\leg
{n}m_j}{N(n)^sN(m)^w} \\
=& |U_K|\sum_{\substack{ m\odd \\ m \equiv 1 \bmod {S_{K,j}} } }\sum_{n\odd}\frac{\leg {n}m_j}{N(n)^sN(m)^w} = |U_K| \sum_{\substack{ m\odd \\ m \equiv 1 \bmod {S_{K,j}} } } \frac{L(s, \chi_{j,m})}{N(m)^w}.
\end{split}
\end{align}

  We now apply \eqref{Aswexp} and the Mellin inversion to see that for $\Re(s)=c$ large enough,
\begin{align}
\label{MellinInversionj}
\begin{split}
\frac {1}{2}\sum_{\chi \in C_K} \sum_{0 \neq n \in \mathcal O_K} L(w, \chi \cdot \chi^{(n)}_{2})\Phi \left( \frac
{N(n)}X \right)=& \frac1{2\pi i}\int\limits_{(c)}A_2(s,w)X^s\widehat\Phi(s) \dif s, \; \mbox{and} \\
\frac {1}{\#h_{(S_{K,j})}}\sum_{\chi \bmod {S_{K,j}}} \sum_{0 \neq n \in \mathcal O_K} L(w, \chi \cdot \chi^{(n)}_{j})\Phi \left( \frac
{N(n)}X \right)=&  \frac1{2\pi i}\int\limits_{(c)}A_j(s,w)X^s\widehat\Phi(s) \dif s,
\end{split}
\end{align}
  where the Mellin transform $\widehat{f}$ of any function $f$ is defined as
\begin{align*}
     \widehat{f}(s) =\int\limits^{\infty}_0f(t)t^s\frac {\dif t}{t}.
\end{align*}

The next few sections are devoted to establishing analytical properties of $A_j(s,w)$.
   
\subsection{First region of absolute convergence of $A_j(s,w)$}
\label{Sec: first region}

  We apply \eqref{L1estimation} and partial summation to treat the sum over $n$ in \eqref{Aswexp} to see that except for a simple pole at $w=1$, all sums of the right-hand side expression in \eqref{Aswexp} are convergent for $\Re(s)>1$, $\Re(w) \geq 1/2, \ |w-1|>\varepsilon$. This further implies that $(w-1)A_j(s,w)$ is holomorphic when $\Re(s)>1$, $\Re(w) \geq 1/2$. \newline

  When $\Re(w) < 1/2$,  we recall that $\chi^{(n)}_{j}$ is a Hecke character of trivial infinite type whenever it is defined.  We write $n=n_1n^j_2$ such that $n_1, n_2$ are primary and $n_1$ is $j$-th power free. We also write $\widehat{\chi_{j,n_1}}$ for the primitive Hecke character that induces $\chi \cdot \chi^{(n_1)}_j$ for any Hecke character $\chi$ of trivial infinite type. It follows from our discussions in Section \ref{sec2.4} that for all values of $j$ under our consideration, the conductor of $\widehat{\chi_{j,n_1}}$ divides $S_{K,j}n_1$ and is divisible by $n_1/(n_1, S_{K,j})$, where $(a, b)$ denotes the primary greatest common divisor of any $a, b \in O_K$. We now have
\begin{align*}
\begin{split}
	\big|L(w,  \chi \cdot \chi^{(n)}_j)\big|=&	\big|\prod_{\varpi | S_{K,j}n_2}(1-\widehat{\chi_{j,n}}(\varpi)N(\varpi)^{-w})\big| \cdot
\big| L(w, \widehat{\chi_{j,n_1}})\big|.
\end{split}
\end{align*}

  We bound the product above in a way similar to that given in \eqref{Lnbound}.  Thus, if $\Re(w) < 1/2$,
\begin{align}
\begin{split}
\label{Abound}
		A_j(s,w)
\ll&  \sum_{\substack{n_2 \odd }}\frac {N(n_2)^{\max (0,-\Re(w))+\varepsilon}
}{N(n_2)^{js}}\sum_{\substack{n_1 \odd}}\Big | \frac{L(w, \widehat{\chi_{j, n_1}})}{N(n_1)^s} \Big | \\
\ll & \sum_{\substack{n_2 \odd }}\frac {N(n_2)^{\max (0,-\Re(w))+\varepsilon}
}{N(n_2)^{js}}\sum_{\substack{n_1 \odd}}\frac {\Gamma(1-w)}{\Gamma(w)}\Big | \frac{L(1- w, \overline{\widehat{\chi_{j, n_1}}})}{N(n_1)^{s+w-1/2}} \Big |,
\end{split}
\end{align}
  where the last bound above comes by using the functional equation \eqref{fneqnL}. \newline

  We note that when $\Re(w) < 1/2$,
\begin{align}
\label{L1wbound}
\begin{split}
		\big |L(1- w, \overline{\widehat{\chi_{j, n_1}}}) \big | \ll &	\big|\prod_{\varpi | S_{K,j}n_1}(1-\overline{\widehat{\chi_{j, n_1}}}(\varpi)N(\varpi)^{-(1-w)})^{-1}\big| \cdot
\big| L(1-w, \widehat{\chi} \cdot \widehat{\chi^{(n_1)}_j)}\big| \ll N(n_1)^{\varepsilon}\big| L(1-w, \widehat{\chi} \cdot \widehat{\chi^{(n_1)}_j)}\big|.
\end{split}
\end{align}

  We apply the above estimation to the last expression in \eqref{Abound}, then we apply \eqref{L1estimation} and partial summation to treat the sum over $n_1$ there to see that all sums of the right-hand side expression in \eqref{Abound} are convergent for $\Re(js)>1$, $\Re(js+w)>1$, $\Re(s+w)>3/2$ and
$\Re(w) < 1/2$.  We then deduce that $(w-1)A_j(s,w)$ is holomorphic in the region
\begin{equation*}
		S_{0,j}=\{(s,w): \Re(s)>1, \ \Re(s+w)>3/2\}.
\end{equation*}

 Next, the last expression of \eqref{Aswexp1} gives that, for $j>2$,
\begin{align} \label{Sum A(s,w,z) over n}		
A_j(s,w)=& |U_K|\sum_{\substack{ m\odd \\ m \equiv 1 \bmod {S_{K,j}} } } \frac{L(s, \chi_{j,m})}{N(m)^w}.
\end{align}

The above sum can be treated in a way analogous to the bound in \eqref{Abound}, upon interchanging the role
of $s$ and $w$ there. Here we point out that by writing $m=m_1m^j_2$ with $m_1, m_2$ being primary and $m_1$ $j$-th power free, the symbol $\chi_{j,m_1}$ is again a Hecke character of trivial infinite type modulo $m_1$ since both $\chi_{j,m}$ and $\chi_{j,m^{j}_2}$ are trivial on $U_K$.    Similar arguments apply to the case $j=2$ as well. Thus, it follows from our discussions above that except for a simple pole at $s=1$ arising from the summands with $m=
\text{a $j$-th power}$, $(s-1)A_j(s,w)$ is holomorphic in the region
\begin{align*}
		S_{1,j}=& \{(s,w): \Re(w)>1, \ \Re(s+w)>3/2\}.
\end{align*}

 Notice that the convex hull of $S_{0,j}$ and $S_{1,j}$ is
\begin{equation*}
		S_{2,j}=\{(s,w):\ \Re(s+w)> 3/2 \}.
\end{equation*}

Now Theorem \ref{Bochner} yields that $(s-1)(w-1)A_j(s,w)$ can be holomorphically continued to the region $S_{2,j}$ for all $j$'s concerned.

\subsection{Residue of $A_j(s,w)$ at $s=1$}
\label{sec:resA}

For $j>2$, we see that $A_j(s,w)$ has a pole at $s=1$ arising from the terms with $m= \text{$j$-th power}$ from \eqref{Sum A(s,w,z) over n}. In order
to compute the corresponding residue for later use, we define the sum
\begin{align*}
\begin{split}
 A_{j,1}(s,w) =: |U_K| \sum_{\substack{m \odd \\ m \equiv 1  \bmod {S_{K,j}} \\ m =  \text{a $j$-th power}}}\frac{\zeta_K(s)\prod_{\varpi |
m}(1-N(\varpi)^{-s}) }{N(m)^w}.
\end{split}
\end{align*}

 For any $t \in \mc$ and any primary element $n \in \mathcal O_K$, let $a_t(n)$ be the multiplicative function defined on the set of primary
elements such that $a_t(\varpi^k)=1-1/N(\varpi)^t$ for any primary prime $\varpi$.  This notation renders
\begin{align*}
\begin{split}
 A_{j,1}(s,w)
=   |U_K|\zeta_K(s)\sum_{\substack{m \odd \\ m \equiv 1  \bmod {S_{K,j}} \\ m =  \text{a $j$-th power}}}\frac{a_s(m) }{N(m)^w}
=\frac {|U_K| \zeta_K(s)}{\#h_{(S_{K,j})}}\sum_{\chi \bmod {S_{K,j}}}\sum_{\substack{m \odd \\ m =  \text{a $j$-th power}}}\frac{\chi(m)a_s(m)
}{N(m)^w} .
\end{split}
\end{align*}

 We recast the last sum above as an Euler product, obtaining
\begin{align} \label{residuesgen}
\begin{split}
	 A_{j,1}(s,w)
=&   \frac {|U_K|\zeta_K(s)}{\#h_{(S_{K,j})}}\sum_{\chi \bmod {S_{K,j}}} \prod_{\varpi \odd}\sum_{\substack{m \geq0}}\frac{\chi(\varpi^{jm})
a_s(\varpi^{jm})}{N(\varpi)^{jmw}} \\
=& \frac {|U_K|\zeta_K(s)}{\#h_{(S_{K,j})}}\sum_{\chi \bmod {S_{K,j}}}\prod_{\varpi \odd}\lz1+\lz 1-\frac 1{N(\varpi)^s} \pz \frac
{\chi(\varpi)^j}{N(\varpi)^{jw}}(1-\chi(\varpi)^jN(\varpi)^{-jw})^{-1} \pz \\
=: &  \frac {|U_K|\zeta_K(s)}{\#h_{(S_{K,j})}}\sum_{\chi \bmod {S_{K,j}}} \frac{L(jw, \chi^j)}{L(s+jw, \chi^j)}.
\end{split}
\end{align}

It follows from \eqref{residuesgen} that except for simple poles at $s=1$ and $w=1/j$, the function $A_{j,1}(s,w)$ is holomorphic in the region
\begin{align} \label{S3}
	S_{3,j}=\Big\{(s,w):\ &  \Re(s+jw)>1 \Big\}.
\end{align}

   Recall that we denote $r_K$ for the residue of $\zeta_K(s)$ at $s = 1$. It follows that for $j>2$,
\begin{align}
\label{Residue at s=1}
 \res_{s=1}& A_j(s, \tfrac{1}{2}+\alpha) = \res_{s=1} A_{j,1}(s, \tfrac{1}{2}+\alpha)
=\frac {|U_K|r_K}{\#h_{(S_{K,j})}}\sum_{\chi \bmod {S_{K,j}}} \frac{L(j(\tfrac{1}{2}+\alpha), \chi^j)}{L(1+j(\tfrac{1}{2}+\alpha), \chi^j)}.
\end{align}

  Similar arguments apply to the case $j=2$ as well. We write for brevity $\psi_1$ for the principal character modulo $B_K$ to note that in this case $\chi^2$ becomes $\psi_1$ so that we have
\begin{align}
\label{Residue at s=1j=2}
\begin{split}
 \res_{s=1} A_2(s, \tfrac{1}{2}+\alpha) =&
|U_K|r_K \frac{\zeta_K(1+2\alpha)}{\zeta_K(2+2\alpha)}\prod_{\varpi | B_K}\Big(1+\frac 1{N(\varpi)^{1+2\alpha}}\Big )\Big(1-\frac
1{N(\varpi)^{2+2\alpha}}\Big )^{-1}.
\end{split}
\end{align}

\subsection{Second region of absolute convergence of $A_j(s,w)$}

 We infer from \eqref{Sum A(s,w,z) over n} that for $j>2$,
\begin{align} \label{A1A2}
 A_j(s,w) = |U_K|\sum_{\substack{m \odd \\ m \equiv 1 \bmod {S_{K,j}} \\ m =  \text{a $j$-th power}}}\frac{L( s, \chi_{j,m})}{N(m)^w}
+|U_K|\sum_{\substack{m \odd \\ m \equiv 1 \bmod {S_{K,j}} \\ m \neq  \text{a $j$-th power}}}\frac{L( s, \chi_{j,m})}{N(m)^w} =: \ A_{j,1}(s,w)+A_{j,2}(s,w).
\end{align}
  We decompose for $A_2(s,w)$ similarly, replacing the condition $m \equiv 1 \bmod {S_{K,j}}$ in $A_{j,1}, A_{j,2}$ above by the condition $(m, B_K)=1, \chi_{2,m}(U_K) =1$ in $A_{2,1}, A_{2,2}$.  Recall from our discussions in the previous section that $A_{j,1}(s,w)$ is holomorphic in the region $S_{3,j}$, except for simple poles at $s=1$ and $w=1$ for all $j$ under our consideration. \newline

  Next, observe that for $m \equiv 1 \bmod S_{K,j}$ and $m$ not a $j$-th power, the Hecke character $\chi_{j,m}$ is modulo $m$ of trivial
infinite type  such that $\chi_{j,m}(0)=0$ and $g_K(0,\chi_{j,m})=0$. We thus apply the functional equation given in Proposition \ref{Functional equation with Gauss sums} for $L\lz s, \chi_{j,m} \pz$ in the case $m \neq \text{a $j$-th power}$ to see that for $j>2$,
\begin{align}
\begin{split}
\label{Functional equation in s}
 A_{j,2}(s,w) =\Big(\frac {2\pi}{\sqrt{|D_K|}}\Big )^{2s-1}\frac {\Gamma(1-s)}{\Gamma(s)} C_j(1-s,s+w),
\end{split}
\end{align}
 where $C_j(s,w)$ is given by the double Dirichlet series
\begin{align}
\label{Cj12}
\begin{split}
		C_j(s,w)=& \sum_{\substack{q, m \in \mathcal O_K \\ q\neq 0, \ m \odd \\ m \equiv 1  \bmod {S_{K,j}} \\ m \neq  \text{a $j$-th
power}}}\frac{g_K(q, \chi_{j,m})}{N(q)^sN(m)^w}=\sum_{\substack{q, m \in \mathcal O_K \\ q\neq 0, \ m \odd \\ m \equiv 1  \bmod {S_{K,j}}
}}\frac{g_K(q, \chi_{j,m})}{N(q)^sN(m)^w}-\sum_{\substack{q, m \in \mathcal O_K \\ q\neq 0, \ m \odd \\ m \equiv 1  \bmod {S_{K,j}} \\ m =
\text{a $j$-th power}}}\frac{g_K(q, \chi_{j,m})}{N(q)^sN(m)^w} \\
=: & \ C_{j,1}(s,w)-C_{j,2}(s,w).
\end{split}
\end{align}	

   Note that the functional equation \eqref{Functional equation in s} and the decomposition \eqref{Cj12} are also valid  for $A_{2,2}(s,w)$ and $C_2(s,w)$, upon replacing the condition $m \equiv 1 \bmod {S_{K,j}}$ appearing in \eqref{Functional equation in s} and \eqref{Cj12} by the condition $(m, B_K)=1, \chi_{2,m}(U_K) =1$ throughout.

\subsubsection{\bf The case $j > 2$}
\label{sec: jlarge}

To estimate of $C_{j,1}(s,w)$ and $C_{j,2}(s,w)$ for the case $j > 2$, we recast $C_{j,1}(s,w)$ as
\begin{align*}
\begin{split}
		C_{j,1}(s,w)=\frac {1}{\#h_{(S_{K,j})}}\sum_{\chi \bmod {S_{K,j}}} \sum_{\substack{q, m \in \mathcal O_K \\ q\neq 0, \ m \odd
}}\frac{\chi(m)g_K(q, \chi_{j,m})}{N(q)^sN(m)^w}=\frac {1}{\#h_{(S_{K,j})}}\sum_{\chi \bmod {S_{K,j}}} \sum_{\substack{q \in \mathcal O_K \\ q\neq 0
}}\frac{h_j(q,w;\chi)}{N(q)^s},
\end{split}
\end{align*}	
 where $h_j(q,w;\chi)$ is defined as in \eqref{h}. \newline

As $\chi$ is a Hecke character modulo $S_{K,j}$ of trivial infinite type, we deduce from Lemma \ref{lem1} that $(w-1-1/j)C_{j,1}(s,w)$ is holomorphic in the region
\begin{align}
\label{Cj1region}
		\{(s,w):\ \Re(s)>1, \ \Re(s+ w/2)> 7/4, \ \Re(w)>1+ 1/j \}.
\end{align}
  Moreover, in the above region we have
\begin{align}
\label{Cj1bound}
		(w-1-1/j)C_{j,1}(s,w) \ll (1+|w|^2)^{ \max \{ (j-1)/2 \cdot (3/2-\Re(w)), 0\}+\varepsilon)}.
\end{align}

   Next, we note that
\begin{align*}
  C_{j,2}(s,w)= \sum_{\substack{q \in \mathcal O_K \\ q\neq 0}}\frac{1}{N(q)^s} \sum_{\substack{ m \in \mathcal O_K \\ m \odd \\ m \equiv 1
\bmod {S_{K,j}}\\ m =  \text{a $j$-th power}}}\frac{g_K(q,\chi_{j,m})}{N(m)^w} =  \sum_{\substack{q \in \mathcal O_K \\ q\neq 0}}\frac{1}{N(q)^s} \sum_{\substack{ m \in \mathcal O_K \\ m \odd \\ m^j \equiv 1  \bmod
{S_{K,j}} }}\frac{g_K(q, \chi_{j,m^j})}{N(m)^{jw}}.
\end{align*}

  Applying the trivial bound that $g_K(q, \chi_{j,m^j})\leq N(m^j)$, we see that $C_{j,2}(s,w)$ is holomorphic in the region $\{ (s, w): \
\Re(s)>1,\ \Re(w)>1+1/j \}$. This together with \eqref{Cj12} and \eqref{Cj1region} implies that $(w-1-1/j)C_j(s,w)$ is defined in the region
given in \eqref{Cj1region}. Combining this with \eqref{S3} and \eqref{A1A2}, we infer that the function $(s-1)(w-1)(s+w-1-1/j)A_j(s,w)$ can be extended to
the region
\begin{align*}
		S_{4,j}=& \{(s,w): \ \Re(s+jw)>1, \ \Re(s+w)>1+1/j, \ \Re(w-s)> 3/2, \ \Re(s)<0 \}.
\end{align*}

  One checks directly that the convex hull of $S_2$ and $S_{4,j}$ is
\begin{equation*}
		S_{5,j}=\{(s,w):\ \Re(s+w)>1+ 1/j \}.
\end{equation*}
Now from Theorem \ref{Bochner}, $(s-1)(w-1)(s+w-1-1/j)A_j(s,w)$ continues holomorphically to the region $S_{5,j}$.

\subsubsection{\bf The case $j =2$}

Note again that $C_2(s,w)$ is initially convergent for $\Re(s)$, $\Re(w)$ large enough.  To extend this region, we note that we have
$G_K(k,\chi_{2, m}) =g_K(k,\chi_{2, m})$ if $(m, B_K)=1$ and $\chi_{2,m}(U_K) =1$. This allows us to recast $C_{2,1}(s,w)$ as
\begin{align}
\label{C1exp}
\begin{split}
  C_{2,1}(s,w)=& \sum_{\substack{q \in \mathcal O_K \\ q\neq 0}}\frac{1}{N(q)^s} \sum_{\substack{m \in \mathcal O_K \\ m \odd \\ (m, B_K)=1 \\ \chi_{2,m}(U_K) =1 }}\frac{g_K(q,\chi_{2,m})}{N(m)^w} = \sum_{\substack{q \in \mathcal O_K \\ q\neq 0}}\frac{1}{N(q)^s} \sum_{\substack{m \in \mathcal O_K \\ m \odd \\ (m, B_K)=1 \\ \chi_{2,m}(U_K) =1 }}\frac{G_K(q,\chi_{2,m})}{N(m)^w} \\
=& \frac {1}{2}\sum_{\chi \in C_K}\sum_{\substack{q \in \mathcal O_K \\ q\neq 0}}\frac{1}{N(q)^s} \sum_{\substack{m \in
\mathcal O_K \\ m \odd  }}\frac{\chi(m) G_K(q, \chi_{2,m})}{N(m)^w}.
\end{split}
\end{align}
  Similarly,
\begin{align}
\label{C21exp}
\begin{split}
  C_{2,2}(s,w)=& \frac {1}{2}\sum_{\chi \in C_K}\sum_{\substack{q \in \mathcal O_K \\ q\neq 0}}\frac{1}{N(q)^s} \sum_{\substack{m \in
\mathcal O_K \\ m \odd \\m = \square }}\frac{\chi(m) G_K(q, \chi_{2,m})}{N(m)^w},
\end{split}
\end{align}
where $m=\square$ means $m$ is a perfect square. \newline

Now using a technique of K. Soundararajan and M. P. Young in \cite[\S 3.3]{S&Y}, we write every $q \in \mathcal O_K, q \neq 0$ uniquely as $q=q_1q^2_2$ with $q_1$ square-free and $q_2$ primary.  So from \eqref{C1exp} and \eqref{C21exp},
\begin{equation} \label{Cidef}
		C_{2,i}(s,w)=\frac {1}{2}\sum_{\chi \in C_K} \sumstar_{q_1}\frac{ D_i(s, w; q_1, \chi)}{N(q_1)^s}, \quad i=1,2,
\end{equation}
where $\sum^*$ stands for the sum over square-free elements in $\mathcal O_K$ and
\begin{align}
\label{Didef}
		D_1(s, w; q_1, \chi)=\sum_{\substack{q_2, l \in \mathcal O_K \\ q_2, l \odd
}}\frac{\chi(l)G_K(q_1q^2_2,\chi_{2,l})}{N(q_2)^{2s}N(l)^w} \quad \mbox{and} \quad D_2(s, w; q_1, \chi)=\sum_{\substack{q_2, l \in \mathcal
O_K \\ q_2, l \odd }}\frac{\chi(l^2)G_K(q_1q^2_2,\chi_{2,l^2})}{N(q_2)^{2s}N(l)^{2w}}.
\end{align}

The following result gives the analytic properties of $D_j(s, w;q_1, \chi)$.
\begin{lemma}
\label{Estimate For D(w,t)}
 With the notation as above and assuming the truth of GRH, the functions $D_j(s, w; q_1, \chi)$, with $j=1$, $2$ have meromorphic
continuations to the region
\begin{align}
\label{Dregion}
		\{(s,w): \Re(s)>1, \ \Re(w)> \tfrac{3}{4} \}.
\end{align}
    Moreover, the only pole in this region is simple and occurs at $w = 3/2$ when $j=1$, $\chi \cdot \chi^{(\eta_K q_1)}_2=\psi_1$.  For $\Re(s) \geq
1+\varepsilon$, $\Re(w) \geq 3/4+\varepsilon$, away from the possible
poles, we have
\begin{align}
\label{Diest}
			|D_j(s, w; q_1, \chi)|\ll (N(q_1)(1+|w|^2))^{\max \{ (3/2-\Re(w))/2, 0 \}+\varepsilon}.
\end{align}		
\end{lemma}
\begin{proof}
   We focus on $D_1(s, w; q_1, \chi)$ here since the proof for $D_2(s, w; q_1, \chi)$ is similar.  Lemma \ref{GausssumMult} gives that the summands  in the double sum in \eqref{Didef} defining $D_1(s,w; q_1, \chi)$ are jointly multiplicative functions of $l,q_2$.  We can write $D_1(s,w; q_1, \chi)$ as an Euler product
\begin{align}
\label{D1Eulerprod}
\begin{split}	
 &	D_1(s, w; q_1, \chi)= \prod_{\varpi} D_{1,\varpi}(s, w; q_1,  \chi),
\end{split}
\end{align}
  where
\begin{align} \label{Dexp}
D_{1,\varpi}(s, w; q_1,  \chi)= \displaystyle
\begin{cases}
\displaystyle \sum_{k=0}^\infty\frac{ 1}{N(\varpi)^{2ks}}, & \varpi |B_K, \\ \\
\displaystyle \sum_{l,k=0}^\infty\frac{ \chi(\varpi^{l})G_K\lz q_1\varpi^{2k}, \chi_{2, \varpi^l}\pz  }{N(\varpi)^{lw+2ks}}, & (\varpi, B_K)=1.
\end{cases}
\end{align}

Further decomposing the above double sum, we get
\begin{align} \label{Dgenest}	
\begin{split}
  &	\sum_{l,k=0}^\infty\frac{ \chi(\varpi^{l})G_K\lz q_1\varpi^{2k}, \chi_{2, \varpi^l}\pz  }{N(\varpi)^{lw+2ks}}  = \sum_{l=0}^\infty \frac{
\chi(\varpi^{l}) G_K\lz q_1, \chi_{2, \varpi^l}\pz  }{N(\varpi)^{lw}}  + \sum_{l \geq 0, k \geq 1}\frac{ \chi(\varpi^{l})G_K\lz
q_1\varpi^{2k}, \chi_{2, \varpi^l}\pz  }{N(\varpi)^{lw+2ks}}.
\end{split}
\end{align}

As $q_1$ is square-free, we deduce from Lemma \ref{GausssumMult} that
\[	|G_K\lz q_1\varpi^{2k}, \chi_{2, \varpi^l} \pz| \ll N(\varpi)^l, \]
and
\[ G_K\lz q_1\varpi^{2k}, \chi_{2, \varpi^l}\pz=0, \; \mbox{if} \; l \geq 2k+3. \]

  The above estimations allow us to see that when $\Re(s)>1$, $\Re(w)>3/4$,
\begin{align}
\label{Dk1est}	
\begin{split}
 \sum_{l \geq 0, k \geq 1} & \frac{\chi(\varpi^{l}) G_K\lz q_1\varpi^{2k}, \chi_{2, \varpi^l}\pz  }{N(\varpi)^{lw+2ks}}
\ll   \sum^{\infty}_{k=1}\sum_{0 \leq l \leq 2k+2}\Bigg| \frac{1}{N(\varpi)^{l(w-1)+2ks}}\Bigg|  \\
\ll & \sum^{\infty}_{k=1}\frac{2k+3}{N(\varpi)^{2k\Re(s)}}\Big (1+ \frac 1{N(\varpi)^{(2k+2)(\Re(w)-1)}}\Big ) \ll N(\varpi)^{-2\Re(s)}+N(\varpi)^{-2\Re(s)-4\Re(w)+4} .
\end{split}
\end{align}
Further applying Lemma \ref{GausssumMult} yields that for $\varpi \nmid B_Kq_1$, $\Re(s)>1$, $\Re(w)>3/4$,
\begin{align}
\label{Dgenl0gen}	
\begin{split}
  \sum_{l=0}^\infty & \frac{ \chi(\varpi^{l}) G_K\lz q_1, \chi_{2, \varpi^l}\pz  }{N(\varpi)^{lw}}  = 1+\frac{ \chi(\varpi) \cdot \chi^{(
\eta_K q_1)}_2(\varpi)}{N(\varpi)^{w-1/2}} \\
& = L_{\varpi} \lz w-\tfrac{1}{2}, \chi \cdot \chi^{(\eta_K q_1)}_2 \pz  \lz 1-\frac {(\chi \cdot \chi^{(\eta_K q_1)}_2)^2(\varpi)}{N(\varpi)^{2w-1}} \pz =  \frac {L_{\varpi}\lz w-\tfrac{1}{2}, \chi \cdot \chi^{(\eta_K q_1)}_2\pz}{L_{\varpi}(2w-1, (\chi \cdot \chi^{(\eta_K q_1)}_2)^2)}.
\end{split}
\end{align}
Now from \eqref{Dexp}--\eqref{Dgenl0gen}, for $\varpi \nmid B_Kq_1$, $\Re(s)>1$, $\Re(w)>3/4$,
\begin{align}
\label{Dgenexp}	
\begin{split}
   & D_{1,\varpi}(s, w; q_1,  \chi) =  \frac {L_{\varpi}\lz w-\tfrac{1}{2}, \chi \cdot \chi^{(\eta_K q_1)}_2\pz}{L_{\varpi}(2w-1, (\chi \cdot \chi^{(\eta_K q_1)}_2)^2)}\lz
1+O \Big ( N(\varpi)^{-2\Re(s)}+N(\varpi)^{-2\Re(s)-4\Re(w)+4} \Big ) \pz .
\end{split}
\end{align}
  We deduce the first assertion of the lemma readily from \eqref{D1Eulerprod}, \eqref{Dexp} and the above.  Moreover, the only pole in region \eqref{Dregion} is simple and located at $w = 3/2$.  This occurs if $\chi \cdot \chi^{(\eta_K q_1)}_2$ is a principal character. \newline

Furthermore, Lemma \ref{GausssumMult} implies that if $\varpi \nmid B_K$ and $\varpi | q_1$,
\begin{align} \label{Dgenl0}	
\begin{split}
  & \sum_{l=0}^\infty \frac{ \chi(\varpi^{l}) G_K\lz q_1, \chi_{2,\varpi^l} \pz  }{N(\varpi)^{lw}} = 1-\frac{
\chi(\varpi^{2})}{N(\varpi)^{2w-1}} =  1+O(N(\varpi)^{-2\Re(w)+1}).
\end{split}
\end{align}

  It follows from \eqref{Dexp}, \eqref{Dk1est} and \eqref{Dgenl0} that for $\varpi \nmid B_K, \varpi | q_1$, $\Re(s)>1, \Re(w)>3/4$,
\begin{align}
\label{Dgenexp1}	
\begin{split}
   & D_{1,\varpi}(s, w;q_1,  \chi)
=  1+O \Big (N(\varpi)^{-2\Re(w)+1}+N(\varpi)^{-2\Re(s)}+N(\varpi)^{-2\Re(s)-4\Re(w)+4}\Big )  .
\end{split}
\end{align}

  We conclude from \eqref{D1Eulerprod}, \eqref{Dexp}, \eqref{Dgenexp} and \eqref{Dgenexp1} that for $\Re(s)\geq 1+\varepsilon$ and $\Re(w)
\geq 3/4+\varepsilon$,
\begin{align*}
\begin{split}
	D_{1}(s, w; q_1, \chi) \ll	& |N(q_1)|^{\varepsilon} \left|\frac {L\lz w-\tfrac{1}{2}, \chi \cdot \chi^{(\eta_K q_1)}_2\pz}{L(2w-1,
(\chi \cdot \chi^{(\eta_K q_1)}_2)^2)} \right|.
\end{split}
\end{align*}

Now under GRH, \eqref{Lchidgeneralbound} and \eqref{PgLest2} yield
\begin{align*}
\begin{split}
	D_{1}(s, w; q_1, \chi) \ll	 (N(q_1)(1+|w|^2))^{\max \{  (3/2-\Re(w))/2, 0 \}+\varepsilon}.
\end{split}
\end{align*}
 This leads to the estimate in \eqref{Diest} and completes the proof of the lemma.
\end{proof}

Applying Lemma \ref{Estimate For D(w,t)} with \eqref{Cj12} and \eqref{Cidef} gives that $(w-3/2)C_2(s,w)$ is defined in the region
\begin{equation*}
		\{(s,w):\ \Re(s)>1, \ \Re(w)>3/4, \ \Re(s+w/2)>7/4 \}.
\end{equation*}
Together with the above, using \eqref{S3} with $j=2$ and \eqref{A1A2} now implies that $(s-1)(w-1)(s+w-3/2)A_2(s,w)$ can be extended to the region
\begin{align*}
		S_{4,2}=& \{(s,w): \ \Re(s+2w)>1, \ \Re(s+w)>3/4,\ \Re(w-s)>3/2,  \ \Re(s)<0\}.
\end{align*}

The convex hull of $S_2$ and $S_{4,2}$ is
\begin{equation*}
		S_{5,2}=\{(s,w):\ \Re(s+w)>3/4 \}.
\end{equation*}
Now Theorem \ref{Bochner} implies that $(s-1)(w-1)(s+w-3/2)A_2(s,w)$ can be continued holomorphically to $S_{5,2}$.

\subsection{Residue of $A_2(s,w)$ at $s=3/2-w$}
\label{sec:resAw}

Using the notation from the proof of Lemma \ref{Estimate For D(w,t)} and mindful of the definition of $C_K$ in \eqref{BKCKdef}, we see that $\chi \cdot \chi^{(\eta_K q_1)}_2=\psi_1$ for any $\chi \in C_K$ implies that $q_1 \in U_K$ so that $N(q_1)=1$. We then deduce from \eqref{Cj12},  \eqref{Cidef} and Lemma \ref{Estimate For D(w,t)} that $C_2(s, w)$ has a pole at $w=3/2$ and
\begin{align}
\label{Cres}
\begin{split}
\res_{w=3/2}C_2(s,w)=& \frac {1}{2}\sum_{\chi \in C_K}\sum_{\substack{ q_1 \in U_K \\ \chi \cdot \chi^{(\eta_K q_1)}_2=\psi_1}}\res_{w=3/2}D_1(s,w;q_1, \chi).
\end{split}
\end{align}

  We now fix a character $\chi \in C_K$ and apply Lemma \ref{GausssumMult} to obtain that for any prime $\varpi \nmid B_K$, 	
\begin{align}
\label{Dk1estpsi1}	
\begin{split}
 \sum_{l \geq 0, k \geq 1}\frac{ \chi(\varpi^{l})G_K\lz q_1\varpi^{2k}, \chi_{2, \varpi^l}\pz  }{N(\varpi)^{3l/2+2ks}} =& \sum_{k \geq
1}\frac {1}{N(\varpi)^{2ks}}\Big (1+\sum^k_{l=1}\frac{ \chi(\varpi^{2k}) \varphi_K(\varpi^{2l}) }{N(\varpi)^{3l}}+\frac {\chi(\varpi^{2k+1})
\leg {\eta_Kq_1}{\varpi} N(\varpi)^{2k+1/2}}{N(\varpi)^{3(2k+1)/2}} \Big ).
\end{split}
\end{align}

  Using $\chi(\varpi)^2=1$ and $\chi \cdot \chi^{(\eta_Kq_1)}_2=\psi_1$, we evaluate the last expression above to see that
\begin{align*}
\begin{split}
 \sum_{l \geq 0, k \geq 1}\frac{ \chi(\varpi^{l})G_K\lz \eta_Kq_1\varpi^{2k}, \chi_{2, \varpi^l}\pz  }{N(\varpi)^{3l/2+2ks}}
= & \Big(1+\frac 1{N(\varpi)} \Big)N(\varpi)^{-2s}(1-N(\varpi)^{-2s})^{-1}.
\end{split}
\end{align*}

  Further using $\chi \cdot \chi^{(\eta_Kq_1)}_2=\psi_1$, we deduce from \eqref{Dgenest}, \eqref{Dgenl0gen} and \eqref{Dk1estpsi1} that for any prime $\varpi \nmid B_K$,
\begin{align}
\label{D1pexp}	
\begin{split}
   D_{1,\varpi}(s, w; q_1,  \chi)=  \zeta_{K,\varpi}(w-1/2)Q_{\varpi}(s, w),
\end{split}
\end{align}
  where
\begin{align}
\label{Qpexp}	
\begin{split}
   Q_{\varpi}(s, w)\Big |_{w=3/2} =& \zeta_{K, \varpi}(2s)\Big(1-\frac {1}{N(\varpi)^{2}} \Big).
\end{split}
\end{align}
	
 We conclude from \eqref{D1Eulerprod}, \eqref{Dexp}, \eqref{D1pexp} and \eqref{Qpexp} that we have
\begin{align}
\label{D1exp}	
\begin{split}
   D_{1}(s, w; q_1,  \chi) =\zeta_K(w-1/2)Q(s,w),
\end{split}
\end{align}
	where
\begin{align}
\label{Qexp}	
\begin{split}
   Q(s,w)\Big |_{w=3/2}=\frac {\zeta_K(2s)}{\zeta_K(2)}\prod_{\varpi | B_K}\Big(1+\frac 1{N(\varpi)}\Big )^{-1}.
\end{split}
\end{align}

   Note that the expression of $Q(s, w)$ is independent of $q_1$ and $\chi$.  Additionally, as $\chi$ varies over the characters in $C_K$, there are exactly $|U_K|$ possible choices of $q_1$ such that $\chi \cdot \chi^{(\eta_Kq_1)}_2=\psi_1$. It follows from this, \eqref{Cres}, \eqref{D1exp} and \eqref{Qexp} that
\begin{align}
\label{Cresexplicit}
		\res_{w=3/2}C_2(s,w)=& \frac {|U_K|r_K}{2}\frac {\zeta_K(2s)}{\zeta_K(2)}\prod_{\varpi | B_K}\Big(1+\frac 1{N(\varpi)}\Big )^{-1},
\end{align}
  where we recall that $r_K$ is denoted for the residue of $\zeta_K(s)$ at $s=1$. \newline
	
 We obtain from \eqref{A1A2}, the functional equation \eqref{Functional equation in s} and \eqref{Cresexplicit} that
\begin{align*}
  \res_{s=3/2-w}A_2(s,w) =& \res_{s=3/2-w}A_{2,2}(s,w) =\frac {|U_K|r_K}{2} \Big(\frac {2\pi}{\sqrt{|D_K|}}\Big )^{2-2w}\frac {\Gamma(w-1/2)}{\Gamma(3/2-w)} \frac
{\zeta_K(2w-1)}{\zeta_K(2)}\prod_{\varpi | B_K}\Big(1+\frac 1{N(\varpi)}\Big )^{-1} .
\end{align*}

Setting $w=1/2+\alpha$ above gives that
\begin{align}
\label{Aress1alpha}
\begin{split}
			&\res_{s=1-\alpha}A_2(s, \tfrac{1}{2}+\alpha) =\frac {|U_K|r_K}{2} \Big(\frac {2\pi}{\sqrt{|D_K|}}\Big )^{1-2\alpha}\frac
{\Gamma(\alpha)}{\Gamma(1-\alpha)} \frac {\zeta_K(2\alpha)}{\zeta_K(2)}\prod_{\varpi | B_K}\Big(1+\frac 1{N(\varpi)}\Big )^{-1} .
\end{split}
\end{align}
	
 Note that, upon taking $\chi$ to be the principal character modulo $q$ in the functional equation \eqref{fneqnL}, we have
\begin{align*}
  \zeta_K(2\alpha)=|D_K|^{1/2-2\alpha}(2\pi)^{4\alpha-1}\frac {\Gamma(1-2\alpha)}{\Gamma (2\alpha)}\zeta_K(1-2\alpha).
\end{align*}

   This enables us to rewrite \eqref{Aress1alpha} as
\begin{align}
\label{Aress1}
\begin{split}
		&\res_{s=1-\alpha}A_2(s,\tfrac{1}{2}+\alpha) =\frac {|U_K|r_K}{2} \Big(\frac {2\pi}{\sqrt{|D_K|}}\Big )^{2\alpha} \frac{\Gamma
(1-2\alpha)\Gamma ( \alpha)}{\Gamma(1-\alpha)\Gamma (2\alpha)}\cdot\frac{\zeta_K(1-2\alpha)}{\zeta_K(2)}\prod_{\varpi | B_K}\Big(1+\frac
1{N(\varpi)}\Big )^{-1}.
\end{split}
\end{align}

\subsection{Bounding $A_j(s,w)$ in vertical strips}
\label{Section bound in vertical strips}
	
  In this section, we estimate $|A_j(s,w)|$ in vertical strips, which will be needed in our proofs of Theorems \ref{Thmfirstmoment} and \ref{Thmfirstmomentjlarge}. \newline
	
  For previously defined regions $S_{i,j}$, we set
\begin{equation*}
		\widetilde S_{i,j}=S_{i, j,\delta}\cap\{(s,w):\Re(s) \geq -5/2,\ \Re(w) \geq 1/2-\delta\},
\end{equation*}
	where $\delta$ is a fixed number with $0<\delta <1/1000$ and $S_{i, j,\delta}= \{ (s,w)+\delta (1,1) : (s,w) \in S_{i,j} \} $. We further
Set
\begin{equation*}
		p_j(s,w)=(s-1)(w-1)(s+w-1-1/j), \quad \widetilde p_j(s,w)=1+|p_j(s,w)|.
\end{equation*}
Observe that $p_j(s,w)A_j(s,w)$ is analytic in the regions under our consideration. \newline

 As described in Section \ref{Sec: first region}, \eqref{L1estimation} and partial summation lead to the bound $|p_j(s,w)A_j(s,w)| \ll \widetilde p_j(s,w)|w|^{1/2+\varepsilon}$ for $\Re(s)>1$, $\Re(w) \geq 1/2$.  We further apply \eqref{Abound} and \eqref{L1wbound} to convert the case $\Re(w) <1/2$ back to the case $\Re(w) >1/2$. Upon applying \eqref{L1estimation} and partial summation again and making use of \eqref{Stirlingratio}, we see that in the region
$\widetilde S_{0,j}$,
\begin{align*}
\begin{split}
      |p_j(s,w)A_j(s,w)| \ll \widetilde p_j(s,w)(1+|w|^2)^{\max \{1/2-\Re(w), 1/4 \}+\varepsilon}.
\end{split}
\end{align*}

   Analogously, we bound the expression for $A_j(s,w)$ given in \eqref{Sum A(s,w,z) over n} for $j>2$ and a similar expression for $A_2(s,w)$.  This yields in the region $\widetilde S_{1,j}$,
\begin{align*}
	|p_j(s,w)A_j(s,w)|\ll \widetilde p_j(s,w)(1+|s|^2)^{\max \{1/2-\Re(s), 1/4\}+\varepsilon}.
\end{align*}
Utlizing the above bounds, Proposition \ref{Extending inequalities} now gives that in the convex hull $\widetilde S_{2,j}$ of
$\widetilde S_{0,j}$ and $\widetilde S_{1,j}$,
\begin{equation} \label{AboundS2}
		|p_j(s,w)A_j(s,w)|\ll \widetilde p_j(s,w) (1+|w|^2)^{\max \{1/2-\Re(w),1/2 \}+\varepsilon}(1+|s|^2)^{3+\varepsilon}.
\end{equation}

   Moreover, we use the estimates in \eqref{Lchidgeneralbound} for $\zeta_K(s)$ (corresponding taking $\chi$ there to be the principal character modulo $1$) and $L(jw, \chi^j)$ to bound $A_{j,1}(s,w)$ in \eqref{residuesgen}.  In the region $\widetilde S_{3,j}$,
\begin{align} \label{A1bound}
		|A_{j,1}(s,w)| \ll (1+|jw|^2)^{\max \{1/2-\Re(jw), (1-\Re(jw))/2, 0\}+\varepsilon}(1+|s|^2)^{\max \{1/2-\Re(s), (1-\Re(s))/2, 0\}+\varepsilon}.
\end{align}

Also, we deduce from \eqref{Cj12}, \eqref{C1exp}--\eqref{Cidef} and Lemma \ref{Estimate For D(w,t)} that, under GRH,
\begin{equation}
\label{Csbound}
		|C_2(s,w)|\ll (1+|w|^2)^{\max \{ (3/2-\Re(w))/2, 0 \}+\varepsilon}
\end{equation}
   in the region
\begin{equation*}
		\{(s,w):\Re(s) \geq 1+\varepsilon, \ \Re(w) \geq 3/4+\varepsilon\}.
\end{equation*}
With \eqref{A1bound} and \eqref{Csbound}, we now apply \eqref{A1A2}, the functional equation \eqref{Functional equation in s} together with \eqref{Stirlingratio} to bound the ratio of the gamma functions.  Consequently, under GRH, in the region
$\widetilde S_{4,2}$ (keeping in mind that $\Re(s+w)>3/2$ in this case),
\begin{align}
\label{AboundS3}
\begin{split}
|p_2(s,w)A_2(s,w)|\ll  \widetilde p_2(s,w)(1+|s+w|^2)^{\max \{(3/2-\Re(s+w))/2, 0\} +\varepsilon}(1+|s|^2)^{3+\varepsilon} \ll  \widetilde p_2(s,w) (1+|w|^2)^{\varepsilon}(1+|s|^2)^{3+\varepsilon}.
\end{split}
\end{align}

  For $j>2$, we note the estimation \eqref{Cj1bound}  for $(w-1-1/j)C_{j,1}(s,w)$ in the region defined by \eqref{Cj1region} and recall from Section \ref{sec: jlarge} that $C_{j,2}(s,w) \ll 1$ in the same region. Again by \eqref{A1A2}, the functional equation \eqref{Functional equation in s} together with \eqref{Stirlingratio} and \eqref{A1bound}, we see that in the region
$\widetilde S_{4,j}$ ($\Re(s+w)>1+1/j$ in this case and $j \leq 6$),
\begin{align} \label{AboundS3j}
\begin{split}
		|p_j(s,w)A_j(s,w)|\ll& \widetilde p_j(s,w)(1+|s+w|^2)^{\max \{(j-1)/2 \cdot (3/2-\Re(s+w))/2, 0\} +\varepsilon}(1+|s|^2)^{3+\varepsilon} \\
\ll & \widetilde p_j(s,w) (1+|w|^2)^{(j-1)(j-2)/(8j)+\varepsilon}(1+|s|^2)^{5+\varepsilon}.
\end{split}
\end{align}

 We now conclude from \eqref{AboundS2},  \eqref{AboundS3}, \eqref{AboundS3j} and Proposition \ref{Extending inequalities} that in the convex hull $\widetilde S_{5,j}$ of $\widetilde S_{2,j}$ and $\widetilde S_{4,j}$, for all $j$'s under our consideration,
\begin{equation}
\label{AboundS4}
		|p_j(s,w)A_j(s,w)|\ll \widetilde p_j(s,w)(1+|w|^2)^{1/2+\varepsilon}(1+|s|^2)^{5+\varepsilon}.
\end{equation}
 The above estimation is valid under GRH for $j=2$ and unconditionally for $j>2$.

\subsection{Completion of the Proofs}

We now evaluate the integrals in \eqref{MellinInversionj} by shifting the line of integration there to $\Re(s)=1/2+1/j-\Re(\alpha)+\varepsilon$ for all $j$'s of our interest.  Note that integration by parts implies that for any rational integer $E \geq 0$,
\begin{align} \label{whatbound}
 \widehat \Phi(s)  \ll  (1+|s|)^{-E}.
\end{align}

The integrals on the new line can be absorbed into the $O$-terms in \eqref{FirstmomentSmoothed} for $j=2$ and \eqref{FirstmomentSmoothedjlarge} for other $j$ upon using \eqref{AboundS4} and \eqref{whatbound}.  We also encounter two simple poles at $s=1$ for all $j$ and $s=1-\alpha$ for $j=2$ in the process. The the corresponding residue at $s=1$ is given in \eqref{Residue at s=1} for $j>2$ and \eqref{Residue at s=1j=2} for $j=2$ while the corresponding residue at $s=1-\alpha$ is given in \eqref{Aress1}. Direct computations now lead to the main terms given in \eqref{FirstmomentSmoothed} and \eqref{FirstmomentSmoothedjlarge}. This completes the proofs of Theorems \ref{Thmfirstmomentatcentral} and \ref{Thmfirstmomentjlarge}.

\vspace*{.5cm}

\noindent{\bf Acknowledgments.}   P. G. is supported in part by NSFC Grant 12471003 and L. Z. by the Faculty Silverstar Grant PS65447 at the
University of New South Wales (UNSW).

\bibliography{biblio}
\bibliographystyle{amsxport}

\end{document}